\documentclass[pdflatex,sn-mathphys-num]{sn-jnl}

\usepackage{graphicx}%
\usepackage{multirow}%
\usepackage{amssymb}
\usepackage{amsmath}
\usepackage{amsfonts}
\usepackage{mathrsfs}
\usepackage{mathtools}
\usepackage[title]{appendix}%
\usepackage{xcolor}%
\usepackage{textcomp}%
\usepackage{manyfoot}%
\usepackage{booktabs}%
\usepackage{listings}%
\usepackage{enumitem}
\usepackage{centernot}
\usepackage{tikz}
\usetikzlibrary{arrows.meta,positioning,fit,calc,decorations.pathreplacing,shapes.multipart,backgrounds}
\pgfdeclarelayer{background}
\pgfsetlayers{background,main}
\usepackage{cleveref}
\usepackage{graphicx}
\usepackage[ruled,vlined]{algorithm2e}
\DontPrintSemicolon                
\SetKwInput{KwData}{Input}
\SetKwInput{KwResult}{Output}
\SetKwInput{KwInit}{Initialization}

\newcommand{\Ord}{\mathcal{O}}
\newcommand{\R}{\mathcal{R}}
\newcommand{\Uad}{\mathcal{U}_{\mathrm{ad}}}
\newcommand{\X}{\mathcal{X}}
\newcommand{\Qt}{Q_T}
\newcommand{\B}{\mathcal{B}}
\theoremstyle{thmstyleone}
\newtheorem{theorem}{Theorem}[section]
\newtheorem{proposition}[theorem]{Proposition}
\newtheorem{lemma}[theorem]{Lemma}

\theoremstyle{definition}

\newtheorem{assumption}[theorem]{Assumption}
\theoremstyle{remark}
\newtheorem{remark}[theorem]{Remark}

\begin{document}

\title{Size-Selective Threshold Harvesting under Nonlocal Crowding and Exogenous Recruitment}

\author[1]{\fnm{Jiguang} \sur{Yu}}\email{jyu678@bu.edu}
\equalcont{These authors contributed equally to this work as co-first authors.}

\author[2]{\fnm{Louis Shuo} \sur{Wang}}\email{wang.s41@northeastern.edu}
\equalcont{These authors contributed equally to this work as co-first authors.}

\author*[3]{\fnm{Ye} \sur{Liang}}\email{ye-liang@uiowa.edu}
\equalcont{These authors contributed equally to this work as co-first authors.}

\affil[1]{\orgdiv{College of Engineering},
  \orgname{Boston University},
  \orgaddress{\city{Boston}, \postcode{02215}, \state{MA}, \country{USA}}}

\affil[2]{\orgdiv{Department of Mathematics},
  \orgname{Northeastern University},
  \orgaddress{\city{Boston}, \postcode{02115}, \state{MA}, \country{USA}}}

\affil[3]{\orgdiv{College of Engineering},
  \orgname{The University of Iowa},
  \orgaddress{\city{Iowa City}, \postcode{52242}, \state{IA},\country{USA}}}

\abstract{
We propose a nonlinear size-structured fishery model for externally recruited
stocks under size-selective harvesting. The population density satisfies a
McKendrick--von Foerster transport equation in which growth and natural
mortality depend on a nonlocal crowding index, while harvesting acts as a
bounded size-dependent mortality control. Unlike standard self-recruiting
formulations, recruitment is prescribed as a lower-boundary inflow, making the
model suitable for enhancement fisheries or analyses conditional on juvenile
input. For the no-harvest baseline, we derive the stationary size profile and
reduce the nonlinear equilibrium problem to a scalar closure equation, proving
existence and uniqueness under a net monotonicity condition. We introduce an
intrinsic replacement index and show why, in this externally forced setting, it
is a viability diagnostic rather than a persistence threshold. A formal
state--adjoint system yields a bang--bang switching rule; under weak coupling
and single crossing, the optimal policy has a threshold structure. Numerical
experiments validate the approximation and sensitivity trends.}


\maketitle

\section{Introduction}\label{sec:intro}

\begin{figure}[htbp]
    \centering
    \includegraphics[width=0.9\linewidth]{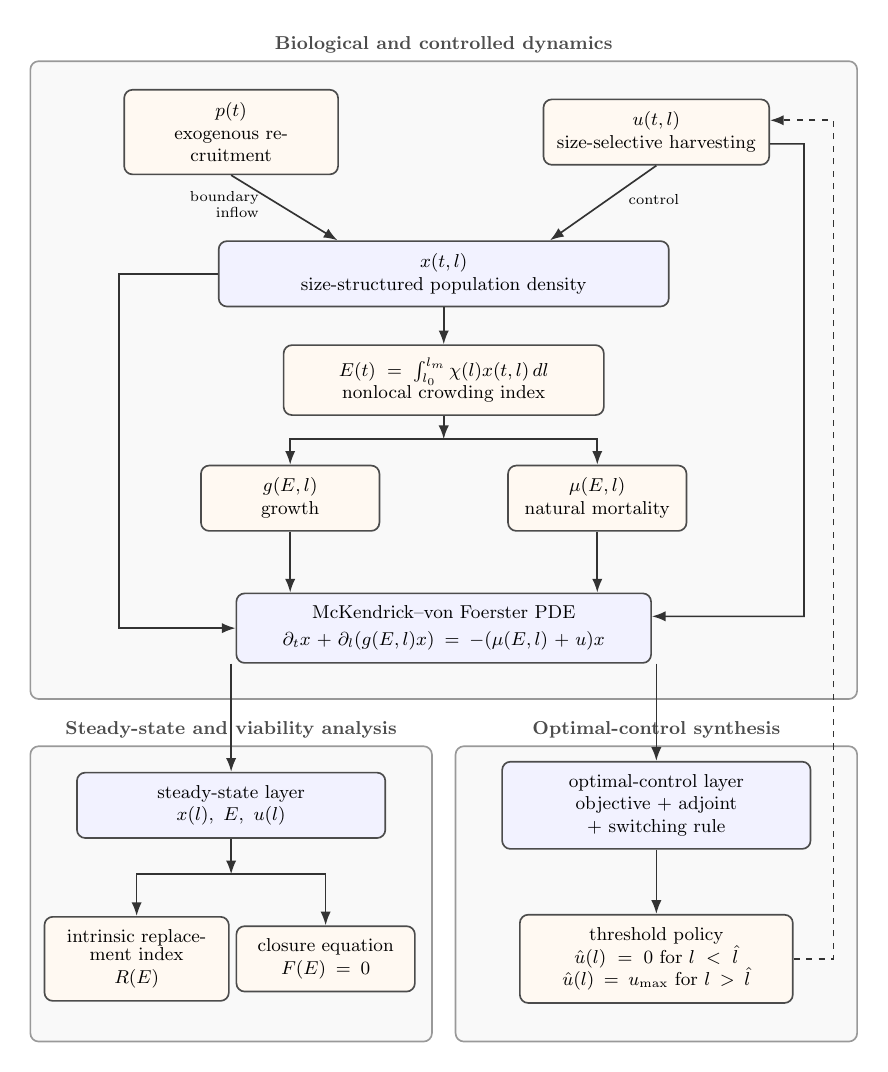}
    \caption{Mechanism diagram for the proposed size-structured fishery model. Exogenous recruitment and harvesting act on the size-structured state \(x(t,l)\). The state determines the crowding index \(E(t)\), which feeds back into growth \(g(E,l)\) and mortality \(\mu(E,l)\). The resulting PDE supports steady-state analysis through the closure equation \(F(E)=0\), biological viability assessment through the intrinsic replacement index \(R(E)\), and optimal-control analysis leading to a threshold harvesting policy.}
    \label{fig:mechanism_diagram}
\end{figure}

Marine fisheries stand at a critical bioeconomic crossroads, where the tension between maximizing economic yield and mitigating ecological collapse demands highly precise, quantitative regulation. Recent global data starkly illustrate the scale of this challenge: according to the 2025 \emph{Review of the State of World Marine Fishery Resources} by the Food and Agriculture Organization, while 64.5\% of assessed marine fish stocks are exploited within biologically sustainable levels, a concerning 35.5\% remain overfished. More alarmingly, the global prevalence of overfishing has continued to creep upward by approximately 1\% per year \cite{markus_fao_2025,wang2025analysis}. Yet, this global average masks the profound efficacy of rigorous, science-based management. In regions governed by stringent quantitative frameworks, such as the Northeast Pacific and Antarctic assessment areas, sustainability rates soar to 92.7\% and 100\%, respectively \cite{markus_fao_2025,wang2025analysis1}. These diverging outcomes provide a powerful mandate for advancing mathematical harvesting models that can tightly couple biological structure, management controls, and long-run system resilience.

The historical necessity for such models is underscored by the catastrophic depletion of heavily exploited stocks. The infamous collapse of the northwest Atlantic cod fishery serves as a canonical warning: excessive and poorly targeted fishing pressure induces not merely a temporary loss of biomass, but a profound structural truncation of the population that severely delays recovery \cite{fao_state_2022,myers_why_1997,hutchings_collapse_2000,liang2025global}. From a mathematical standpoint, such episodes prove that aggregate stock measures—treating a population as an undifferentiated mass—are fundamentally inadequate for optimal management design. Long-run viability is dictated by the demographic composition of the stock, the timing of removals, and the selectivity of the harvest. The central bioeconomic task is therefore to formulate harvesting policies that exploit the structural heterogeneity of a population efficiently without breaching ecological boundaries.

In applied fisheries management, these biological realities are typically addressed through total allowable catch limits, spatial closures, gear restrictions, and minimum landing sizes \cite{agardy_mind_2011,broadhurst_estimating_2006,yu2026beyond,costello_can_2008,hilborn_quantitative_1992,wang2026algebraic}. While indispensable and highly effective in well-managed systems \cite{hilborn_effective_2020,wang2026damage}, these conventional instruments are not, inherently, solutions to a selective optimal control problem. Aggregate catch limits, for instance, cap total removals but fail to dictate which specific individuals should be extracted. This is a critical mathematical oversight because commercial value, baseline survival, growth potential, and reproductive contribution are all highly sensitive to body size \cite{birkeland_importance_2005}. Consequently, there is a pressing need for structured bioeconomic models that treat selective harvesting as a precise mathematical control acting on a distributed population state.

Among possible structuring variables, body size is uniquely suited to fisheries optimization.
Size dictates a fish's commercial market price, its vulnerability to specific gear, and its fecundity;
in many species, larger individuals contribute disproportionately to spawning output \cite{yildiz_bio-economic_2023,yu2026rigorous,hixon_boffffs_2014,barneche_fish_2018}. 
Harvesting policies that disproportionately target these large individuals risk truncating the spawning structure and permanently diminishing stock resilience \cite{anderson_why_2008}. Size-structured models, therefore, naturally capture the fundamental bioeconomic trade-off: deciding not only how much to harvest, but precisely which size classes to target \cite{guiet_boatsv2_2024,xia_multispecies_2021}.

The mathematical foundation for this viewpoint lies in the theory of structured population dynamics. 
This theory traces its origins back to the McKendrick--von Foerster transport equation and its size-structured extensions \cite{mkendrick_applications_1925,banasiak_mckendrickvon_2025,sinko_new_1967}.
Through this mathematical framework, it formalizes how growth, mortality, and recruitment interact via the internal distribution of individuals across physiological or demographic states. 
More general frameworks for physiologically and size-structured populations were subsequently developed by Metz, Diekmann, Webb, and others \cite{metz_dynamics_1986,webb_theory_1985,barril_formulation_2022,andrusyak_mathematical_2024,clement_bifurcation_2025}. 
A particularly important extension for resource management is the incorporation of nonlocal feedback, whereby individual-level coefficients depend on aggregate characteristics of the full population rather than on local state alone \cite{goetz_forest_2010}.
In fisheries terms, such feedback can represent crowding, competition for food, habitat pressure, or other forms of density dependence acting through a population-level environmental index \cite{hritonenko_maximum_2008,hritonenko_bangbang_2009}.

Coupling this biological realism with optimal harvesting theory—rooted in the classical works of Gordon, Scott, and Clark \cite{gordon_economic_1954,scott_fishery_1955,clark_economics_1975,clark_mathematical_1976,clark_worldwide_2007}—presents significant analytical challenges. While single-biomass models are highly tractable \cite{schaefer_aspects_1991,spence_most_1975}, structured models prove that selective exploitation fundamentally alters the geometry of optimal policies \cite{beverton_dynamics_1993,reed_optimum_1980,tahvonen_economics_2009,skonhoft_optimal_2012}. Addressing these distributed-parameter systems requires infinite-dimensional variational methods and Pontryagin-type principles to derive first-order necessary conditions \cite{anita_analysis_2000,wang2026breakdown,kakumani_optimal_2026,yi_age-dependent_2025}. 
In fisheries science, these optimal-control ideas interface naturally with the Baranov formulation of fishing mortality, which treats harvest as an instantaneous mortality component added to natural losses  \cite{quinn_quantitative_1999,hilborn_quantitative_1992,yu2026fullcovariance}.
More recent studies have further examined long-run and turnpike properties of optimal harvesting in structured systems \cite{hritonenko_optimization_2006}, forest management problems \cite{cominetti_asymptotic_2009}, and population models \cite{djema_turnpike_2021}.

Density dependence is indispensable in this setting. Beyond classical stock--recruitment mechanisms such as Beverton--Holt and Ricker \cite{beverton_dynamics_1993,ricker_stock_1954,yu2026chemotactic}, fish populations commonly exhibit density-dependent growth, mortality, and maturation 
\cite{gao2022rolling,tuffley_density-dependence_2024,yu2026microscopic,accolla_density-dependent_2024,horbowy_effects_2025,wang2025multi}. 
In structured PDE models, these effects can be represented naturally through (nonlocal) environmental variables summarizing aggregate crowding or competition  \cite{de_roos_simplifying_2008,ni_optimal_2023,ainseba_population_2022}.
Another recurrent structural feature of such control problems is the emergence of bang--bang policies \cite{opmeer_optimal_2025}. 
In fisheries applications, they are naturally defined by an explicit threshold in minimum harvest size \cite{lavin_minimum_2021,clark_mathematical_1976,cai2026optimal,hritonenko_structure_2007}. 
However, a transparent analytical derivation of such threshold policies for nonlinear size-structured partial differential equations (PDEs) featuring nonlocal density dependence remains a significant gap in the literature.

This paper bridges that gap. We formulate and analyze an infinite-horizon optimal harvesting problem for a nonlinear size-structured fish population governed by a nonlocal transport equation. The state variable $x(t,l)$ represents the density of individuals of size $l$ at time $t$, and the control $u(t,l)$ dictates the size-selective harvest mortality. The population's environment is dynamically modeled through a nonlocal crowding index $E(t)$, which governs both the transport velocity (growth) and natural mortality. 
Our approach to recruitment further underscores the originality of our framework. Rather than relying on a traditional endogenous birth law, recruitment is strictly imposed through an exogenous inflow at the lower boundary. This modeling choice is appropriate for the dynamics of enhancement fisheries or heavily subsidized stocks where early-life survival is decoupled from adult biomass. Consequently, stationary regimes in our model represent operating points sustained by external inflow, and we introduce a novel intrinsic replacement capacity index that serves as a biological viability diagnostic rather than a standard, autonomous persistence threshold. 
However, the model should not be interpreted as a complete stock-assessment
model for a wild self-recruiting population. It is designed for fisheries in
which the inflow of recruits can be regarded as externally prescribed, or for
situations where the management objective is to study size-selective harvesting
conditional on a given recruitment input.

The main contributions of the paper are the following.
\begin{enumerate}[label=(\roman*)]

\item We formulate an original nonlinear size-structured PDE model that integrates nonlocal crowding, size-dependent growth and mortality, size-selective harvesting control, and strictly exogenous lower-boundary recruitment (Section~\ref{sec:model}).

\item We characterize the stationary profile under constant inflow and derive a scalar nonlinear closure equation for the crowding level, proving the existence and uniqueness of the steady state based on explicit continuity, positivity, and monotonicity conditions (Section~\ref{sec:steady_state}).

\item We introduce an intrinsic replacement index tailored to exogenously forced systems, demonstrating its utility as a biological viability diagnostic rather than a sharp persistence threshold (Section~\ref{subsec:replacement}).

\item We derive formal first-order necessary conditions of the Pontryagin type for an infinite-horizon discounted harvesting problem, yielding a coupled state–adjoint system and a pointwise switching rule for the control (Section~\ref{sec:objective}).

\item By applying a weak-coupling reduction to the adjoint equation and imposing an explicit single-crossing condition on the reduced switching function, we derive a conditional threshold characterization of the size-selective bang--bang harvesting policy (Section~\ref{sec:objective}).

\end{enumerate}

Figure~\ref{fig:mechanism_diagram} summarizes the mechanism structure of the model. The remainder of the paper is organized as follows. 
Section~\ref{sec:model} introduces the biological setting and presents the nonlinear size-structured PDE. 
Section~\ref{sec:steady_state} develops the no-harvest stationary analysis under constant external inflow, defines the intrinsic replacement index, and clarifies its interpretation in the exogenously forced setting. 
Section~\ref{sec:objective} formulates the optimal control problem,
derives the formal adjoint system, and establishes the threshold structure of the reduced optimal harvesting policy. 
Section~\ref{sec:numerical} then illustrates the analytical results in a calibrated case study, and the final section discusses limitations, fisheries implications, and possible extensions.

\section{Biological Setting and Governing Size-Structured PDE Model}\label{sec:model}

Table~\ref{tab:notation} provides a summary of notations in this work. 
\begin{table}[h]\centering
\caption{Principal notation.}\label{tab:notation}
\begin{tabular}{ll}
\toprule
Symbol & Meaning\\
\midrule
$x(t,l)$ & size-density (ind.\ cm$^{-1}$)\\
$E(t)=\int\chi x$ & nonlocal crowding index\\
$g(E,l),\mu(E,l)$ & growth speed, natural mortality\\
$u(t,l)\in[0,u_{\max}]$ & size-selective harvest control\\
$p(t)$ & exogenous boundary inflow\\
$R(E)$ & intrinsic replacement index\\
$\hat u,\hat x, \hat E, \hat l$ & value under optimal policy or optimal value \\
$u^*, x^*,E^*, N^*$ & stationary value\\
$\lambda(t,l)$ & adjoint / shadow value\\
$S=c-\lambda$ & reduced switching function; threshold $\hat E(t)$\\
$J,J_T,\widehat J_T$ & infinite-horizon, truncated, tail-corrected revenue\\
\bottomrule
\end{tabular}
\end{table}

Let $l\in[l_0,l_m]$ denote fish body length, where $l_0>0$ is the entry size and $l_m$ is the maximum considered size. Let
\[
x(t,l)\ge 0
\]
denote the population density at time $t\ge 0$ and size $l$, 
so that
$x(t,l)\,dl$ represents the number of individuals in the size interval
$(l,l+dl)$. The total population number is
\begin{equation}\label{eq:totalN}
N(t)\coloneqq \int_{l_0}^{l_m} x(t,l)\,dl.
\end{equation}

To capture the fact that competition for food, space, or oxygen depends on the
entire population composition, we introduce the nonlocal crowding index
\begin{equation}\label{eq:Edef}
E(t)\coloneqq \int_{l_0}^{l_m}\chi(l)\,x(t,l)\,dl,
\end{equation}
where $\chi(l)\ge 0$ is a weighting kernel. The choice $\chi(l)\equiv 1$
yields $E(t)=N(t)$, 
so that every individual contributes equally to the
environment. More generally, kernels such as $\chi(l)\propto l^\beta$ ($\beta>0$) allow
larger individuals to contribute more strongly to crowding, which is
appropriate when resource use or habitat demand scales with body size.

The individual growth rate along the size axis is denoted by $g(E,l)>0$. It
depends on the current crowding level $E$ and on body size $l$, 
reflecting the
biological observation that growth may slow under stronger density-dependent
competition. The natural mortality rate $\mu(E,l)\ge 0$ captures environmental
stress induced by crowding, and we impose the monotonicity condition
\begin{equation}\label{eq:mumono}
\frac{\partial \mu}{\partial E}(E,l)\ge 0,
\end{equation}
meaning that higher crowding does not decrease death risk and may increase it
through mechanisms such as starvation, disease transmission, or oxygen
depletion.

The harvesting control is denoted by $u(t,l)\in[0,u_{\max}]$, representing
size-selective fishing mortality induced by regulatory and operational
decisions, such as mesh size, targeting strategy, and seasonal closures. The
upper bound $u_{\max}$ represents the maximum admissible harvesting intensity.

The inflow of new individuals at the minimum size is modeled through the
boundary condition
\begin{equation}\label{eq:model_bc}
g(E(t),l_0)\,x(t,l_0)=p(t),
\end{equation}
where $p(t)\ge 0$ is an exogenously prescribed recruitment flux. In
fishery applications, $p(t)$ may represent hatchery stocking, the output of
an upstream nursery habitat, or an effective recruitment input summarizing
early-life survival processes that are not explicitly resolved by the model.

The key modeling distinction is that $p(t)$ is not generated
endogenously by the current adult stock through a birth law of the form
\[
g(E(t),l_0)\,x(t,l_0)=\int_{l_0}^{l_m} m(l)\,x(t,l)\,dl.
\]
This formulation is appropriate for enhancement fisheries and other settings in
which recruitment is externally controlled, but it affects the interpretation
of reproduction-type quantities, as discussed in
Section~\ref{subsec:replacement}.

Combining the above ingredients, 
the fish population dynamics are
governed by the following size-structured transport equation with nonlocal
coefficients:
\begin{equation}\label{eq:fish_model}
\left\{
\begin{aligned}
\frac{\partial x(t,l)}{\partial t}
+\frac{\partial}{\partial l}\!\Big(g(E(t),l)\,x(t,l)\Big)
&= -\Big(\mu(E(t),l)+u(t,l)\Big)x(t,l),
&& t>0,\; l\in(l_0,l_m), \\[4pt]
g(E(t),l_0)\,x(t,l_0) &= p(t),
&& t>0, \\[4pt]
x(0,l) &= \phi(l),
&& l\in[l_0,l_m].
\end{aligned}
\right.
\end{equation}
Here the crowding index $E(t)$ is defined by \eqref{eq:Edef}.

Equation \eqref{eq:fish_model} is a nonlinear McKendrick--von Foerster type
balance law in which transport along the size axis is coupled to the population
distribution through the nonlocal environment variable $E(t)$. Such structured
transport equations are standard in population dynamics; see
\cite{cushing_introduction_1998,diekmann_steady-state_2003,wang2026elliptic} for general
structured-population theory and
\cite{hritonenko_maximum_2008,hritonenko_bangbang_2009,hritonenko_mathematical_2013,ainseba_population_2022,liu2025bidirectional}
for optimal-control formulations in age- and size-structured settings.

We collect the standing regularity and structural assumptions used throughout the
paper.

\begin{assumption}[Standing assumptions]\label{ass:standing}
\begin{enumerate}[label=(H\arabic*)]
\item\label{H:reg}
The functions $g(E,l)$ and $\mu(E,l)$ are continuous on
$[0,\infty)\times[l_0,l_m]$ and continuously differentiable with respect to $E$.

\item\label{H:gpos}
There exists a constant $g_{\min}>0$ such that
\[
g(E,l)\ge g_{\min}
\qquad \text{for all } (E,l)\in[0,\infty)\times[l_0,l_m].
\]

\item\label{H:mupos}
\[
\mu(E,l)\ge 0
\qquad \text{for all } (E,l)\in[0,\infty)\times[l_0,l_m].
\]

\item\label{H:mumono}
\[
\partial_E\mu(E,l)\ge 0
\qquad \text{for all } (E,l)\in[0,\infty)\times[l_0,l_m].
\]
That is, increased crowding does not decrease natural mortality.

\item\label{H:chi}
The weighting kernel $\chi(l)$ is continuous on $[l_0,l_m]$, satisfies
$\chi(l)\ge 0$ for all $l\in[l_0,l_m]$, and is not identically zero.

\item\label{H:uadm}
The harvesting control $u$ is measurable and satisfies
\[
0\le u(t,l)\le u_{\max}
\qquad \text{for all } (t,l)\in[0,\infty)\times[l_0,l_m].
\]

\item\label{H:ic}
The initial condition satisfies $\phi\in L^1([l_0,l_m])$ and
$\phi(l)\ge 0$ for a.e.\ $l\in[l_0,l_m]$.
\end{enumerate}
\end{assumption}

In \eqref{eq:fish_model}, the term
\[
\partial_l\bigl(g(E,l)\,x(t,l)\bigr)
\]
represents deterministic transport along the size axis at rate $g(E,l)$.
The term $\mu(E,l)\,x(t,l)$ represents natural mortality, while
$u(t,l)\,x(t,l)$ represents harvest mortality. The nonlocal feedback enters
through the dependence of $g$ and $\mu$ on the aggregate crowding index $E(t)$,
which couples the dynamics of each size class to the population as a whole.

\section{Steady-State Regimes Under Constant External Inflow (No-Harvest Baseline)}
\label{sec:steady_state}

This section studies time-independent operating regimes of the model
\eqref{eq:fish_model} in the absence of harvesting. In fisheries management, a
steady state corresponds to a stationary operating point: the size distribution
is time-independent, the crowding level is constant, and the recruitment inflow
does not vary with time. Such regimes serve as reference points for regulation
and stocking design.
For completeness, we first record the stationary profile
associated with a prescribed harvesting intensity $u(l)$, but the closure
equation and the existence--uniqueness result below are established only for the
no-harvest case $u(l)\equiv 0$.

\subsection{Definition of stationary regime}\label{subsec:ss_def}

A steady state in the stationary regime is defined by
\begin{equation}\label{eq:steadystate_ansatz}
x(t,l)\equiv x(l),\qquad E(t)\equiv E,\qquad u(t,l)\equiv u(l),
\end{equation}
where $x(\cdot)$ and $u(\cdot)$ are time-independent. Substituting
\eqref{eq:steadystate_ansatz} into the boundary condition \eqref{eq:model_bc}
shows that $g(E,l_0)\,x(l_0)$ is constant, so stationarity requires
\begin{equation}\label{eq:p_constant}
p(t)\equiv p=\mathrm{const}.
\end{equation}

Under the ansatz \eqref{eq:steadystate_ansatz}, the state equation
\eqref{eq:fish_model} reduces to the first-order ordinary differential equation
\begin{equation}\label{eq:ss_ode_general}
\frac{d}{dl}\Big(g(E,l)\,x(l)\Big)
=
-\big(\mu(E,l)+u(l)\big)\,x(l),
\qquad l\in[l_0,l_m],
\end{equation}
subject to the stationary boundary inflow
\[
g(E,l_0)\,x(l_0)=p.
\]

For fixed crowding level $E$, constant inflow $p$, and prescribed stationary
harvesting profile $u(l)$, the unique solution of
\eqref{eq:ss_ode_general} is
\begin{equation}\label{eq:ss_profile}
x(l)=\frac{p}{g(E,l)}
\exp\!\left(
-\int_{l_0}^{l}\frac{\mu(E,\xi)+u(\xi)}{g(E,\xi)}\,d\xi
\right),
\qquad l\in[l_0,l_m].
\end{equation}
In the no-harvesting case
$u(l)\equiv 0$, it reduces to
\begin{equation}\label{eq:ss_profile_noharvest}
x(l)=\frac{p}{g(E,l)}
\exp\!\left(
-\int_{l_0}^{l}\frac{\mu(E,\xi)}{g(E,\xi)}\,d\xi
\right),
\qquad l\in[l_0,l_m].
\end{equation}
Stationary profiles of this form are standard in structured population theory.

In the nonlinear model, the crowding index $E$ must satisfy the
self-consistency relation \eqref{eq:Edef}. A commonly used size-weighted choice
is
\begin{equation}\label{eq:chi_l2}
\chi(l)=\chi\,l^2,\qquad \chi>0,
\end{equation}
so that
\begin{equation}\label{eq:closure_E}
E=\chi\int_{l_0}^{l_m} l^2\,x(l)\,dl.
\end{equation}
In the no-harvesting case, substituting \eqref{eq:ss_profile_noharvest} into
\eqref{eq:closure_E} yields the scalar nonlinear closure equation
\begin{equation}\label{eq:F_equation}
F(E)\coloneqq
E-\chi\int_{l_0}^{l_m}
\frac{p\,l^2}{g(E,l)}
\exp\!\left(
-\int_{l_0}^{l}\frac{\mu(E,\xi)}{g(E,\xi)}\,d\xi
\right)\,dl
=0.
\end{equation}

\subsection{Existence and uniqueness of the stationary crowding level}
\label{subsec:exist_unique}

\begin{proposition}[Existence and uniqueness]\label{prop:exist_unique}
Assume that:
\begin{enumerate}[label=(A\arabic*)]
\item\label{A:pchi}
$ p>0 $ and $ \chi>0 $;

\item\label{A:reg}
$g(E,l)$ and $\mu(E,l)$ are continuous on
$[0,\infty)\times[l_0,l_m]$, with
\[
g(E,l)\ge g_{\min}>0,
\qquad
\mu(E,l)\ge 0;
\]

\item\label{A:hazard}
\[
\partial_E\mu(E,l)\ge 0
\qquad
\text{for all } (E,l)\in[0,\infty)\times[l_0,l_m];
\]

\item\label{A:mono}
for each fixed $l\in[l_0,l_m]$, the integrand in \eqref{eq:F_equation} is
non-increasing in $E$, and therefore the map
\[
E\longmapsto
\int_{l_0}^{l_m}
\frac{p\,l^2}{g(E,l)}
\exp\!\left(
-\int_{l_0}^{l}\frac{\mu(E,\xi)}{g(E,\xi)}\,d\xi
\right)\,dl
\]
is non-increasing on $[0,\infty)$.
\end{enumerate}
Then equation \eqref{eq:F_equation} admits a unique solution $\hat E>0$.
\end{proposition}

\begin{proof}
Define
\[
G(E)\coloneqq
\chi\int_{l_0}^{l_m}
\frac{p\,l^2}{g(E,l)}
\exp\!\left(
-\int_{l_0}^{l}\frac{\mu(E,\xi)}{g(E,\xi)}\,d\xi
\right)\,dl,
\]
so that $F(E)=E-G(E)$.
By assumption \ref{A:reg} and the dominated convergence theorem, $G$ is
continuous on $[0,\infty)$, hence so is $F$.

Because $p>0$, $\chi>0$, and the integrand defining $G(E)$ is strictly positive,
we have $G(0)>0$, and therefore
\[
F(0)=-G(0)<0.
\]
Moreover, since $g(E,l)\ge g_{\min}$ and the exponential factor is bounded above
by $1$,
\[
0\le G(E)\le
\chi\int_{l_0}^{l_m}\frac{p\,l^2}{g_{\min}}\,dl
\eqqcolon C<\infty.
\]
Hence
\[
F(E)=E-G(E)\ge E-C\to+\infty
\qquad\text{as } E\to\infty.
\]
By the intermediate value theorem, there exists at least one
$\hat E>0$ such that $F(\hat E)=0$.

Finally, assumption \ref{A:mono} implies that $G$ is non-increasing on
$[0,\infty)$, so $F(E)=E-G(E)$ is strictly increasing. Therefore $F$ can cross
zero at most once, and the root $\hat E$ is unique.
\end{proof}

The unique stationary crowding level $\hat E$ is obtained by solving
\eqref{eq:F_equation}. The corresponding stationary size distribution in the
no-harvesting case is then recovered by substituting $E=\hat E$ into
\eqref{eq:ss_profile_noharvest}. 
This profile represents a long-run baseline
regime of the fishery under constant external inflow $p$ and no harvesting, and
it provides a reference configuration against which management interventions may
be assessed.

\subsection{Preliminary analysis of the endogenous-recruitment closure}\label{subsec:endo}
Replacing the exogenous inflow by the endogenous boundary law
\begin{equation}\label{eq:endo_bc}
g(E,l_0)x(l_0)=\B(S),\qquad S=\int_{l_{\mathrm{mat}}}^{l_m}m(l)x(l)\,dl,
\end{equation}
with $\B\in C^1$, $\B(0)=0$, $\B'(0)=b>0$, and survival
$\Pi(E,l)=\exp(-\int_{l_0}^l\mu/g)$, the stationary profile is
$x(l)=\B(S)\Pi(E,l)/g(E,l)$, and self-consistency gives
\begin{equation}\label{eq:endo_closure}
S=\B(S)R(E),\qquad E=\B(S)Q(E),\qquad
Q(E)=\chi\!\int_{l_0}^{l_m}\frac{l^2\Pi(E,l)}{g(E,l)}\,dl .
\end{equation}
where $R$ is given in \eqref{eq:RE}.
\begin{proposition}[Trivial state, bifurcation, sharp threshold]\label{prop:endo}
The trivial state $S=0$ (with $E=0$) always solves \eqref{eq:endo_closure}, and a
nontrivial stationary stock bifurcates from it iff $\ b\,R(0)>1$. If $\B$ is
Beverton--Holt, $\B(S)=bS/(1+S/\kappa)$, with $E\mapsto R(E),Q(E)$ non-increasing,
the nontrivial state is unique when it exists; the Ricker form
$\B(S)=bSe^{-S/\kappa}$ may give overcompensation/multiplicity.
\end{proposition}
\begin{proof}[Sketch]
With $f(S)=\B(S)R(E(S))-S$ ($E(S)$ solving the second relation, increasing in
$\B(S)$ by monotonicity of $Q$): $f(0)=0$, $f'(0)=bR(0)-1$, and $f\to-\infty$
(saturating $\B$). A positive root exists iff $f'(0)>0$; strict monotonicity gives
uniqueness.
\end{proof}
\begin{remark}
Under \eqref{eq:endo_bc} and Proposition~\ref{prop:endo}, $\R=b R(0)$ regains the sharp dichotomy
$\R>1\Rightarrow$persistence, $\R<1\Rightarrow$extinction near the trivial state.
In the exogenously forced model it is only a viability diagnostic
(Section~\ref{subsec:replacement}). We list in Table~\ref{tab:applicability} the applicability of the exogenous-recruitment model. The applicable scenarios include hatchery/stock-enhancement and ranched
“put–grow–take” fisheries, subsidized or externally managed recruitment, and short-horizon size-
selectivity studies conditional on a given recruitment input. The inapplicable scenarios include wild self-
recruiting stocks where recruitment is spawning biomass-driven, collapse/rebuilding and recruitment-overfishing questions, and depensation/Allee regime.
\end{remark}

\begin{table}[t]\centering
\caption{Applicability of the exogenous-recruitment model.}
\label{tab:applicability}
\begin{tabular}{p{0.46\linewidth}p{0.46\linewidth}}
\toprule
\textbf{Applicable (exogenous inflow appropriate)} & \textbf{Inapplicable (endogenous law needed)}\\
\midrule
Hatchery / stock-enhancement, ranched ``put--grow--take'' fisheries & Wild self-recruiting stocks (recruitment driven by spawning biomass)\\
Subsidized or externally managed recruitment input & Recruitment-overfishing, collapse and rebuilding dynamics\\
Pre-recruit dynamics resolved outside the modeled size range & Depensation / Allee regimes at low stock\\
Short-horizon size-selectivity studies conditional on given recruitment & Long-run persistence/extinction from internal demography\\
\bottomrule
\end{tabular}
\end{table}

\subsection{Intrinsic Replacement Index Under Exogenous Recruitment}
\label{subsec:replacement}

To assess the stock's intrinsic replacement capacity, we introduce a size-dependent
fertility function $m(l)\ge 0$ and define the intrinsic replacement index at
crowding level $E$ by
\begin{equation}\label{eq:RE}
R(E)\coloneqq
\int_{l_0}^{l_m}
\frac{m(l)}{g(E,l)}
\exp\!\left(
-\int_{l_0}^{l}\frac{\mu(E,\xi)}{g(E,\xi)}\,d\xi
\right)\,dl.
\end{equation}

The factor
\[
\exp\!\left(
-\int_{l_0}^{l}\frac{\mu(E,\xi)}{g(E,\xi)}\,d\xi
\right)
\]
is the probability of surviving from size $l_0$ to size $l$ under crowding level $E$,
while $dl/g(E,l)$ is the time an individual spends in the size interval $[l,l+dl]$.
Hence, the integrand in \eqref{eq:RE} represents the expected recruitment
contribution generated while an individual occupies size $l$.
Therefore, integrating over $[l_0,l_m]$ yields the expected total contribution over its lifetime in environment $E$.

Quantities of this type are standard in structured population theory, but their threshold interpretation depends on how recruitment enters the model. 
Because the present model prescribes the boundary inflow externally through
\[
g(E(t),l_0)\,x(t,l_0)=p(t),
\]
the quantity $R(E)$ should be interpreted as an intrinsic replacement index of the adult stock in environment $E$ (representing the lifetime spawning output per recruit), rather than a dimensionless stand-alone persistence or extinction threshold.

To obtain a dimensionless generation multiplier (analogous to the basic reproduction number $\mathcal{R}_0$), $R(E)$ must be coupled with the early-stage survival rate. A sharp threshold interpretation requires an endogenous recruitment boundary law closing the life cycle, for example
\[
g(E(t),l_0)\,x(t,l_0)
=
\mathcal{B}\left(\int_{l_0}^{l_m} m(l)\,x(t,l)\,dl\right),
\]
where the inflow is generated by the adult stock itself. If the early-stage survival rate from spawned output to recruit near the trivial state is denoted by $b = \mathcal{B}'(0) > 0$, the dimensionless threshold condition becomes:
\[
b R(E)>1 \;\Rightarrow\; \text{persistence},
\qquad
b R(E)<1 \;\Rightarrow\; \text{extinction}.
\]
Consequently, for the dimensional index $R(E)$, the critical threshold is exactly $1/b$. 

Despite the absence of a rigid bifurcation threshold under exogenous recruitment settings, $R(E)$ remains a physically real and biologically meaningful diagnostic quantity. It captures the actual expected lifetime spawning output of an individual under the environmental crowding $E$. Qualitatively, a larger $R(E)$ indicates a highly viable population that is robust and capable of thriving in environment $E$, whereas a smaller $R(E)$ signifies an intrinsically vulnerable stock leaning toward extinction.

Accordingly, it is natural to establish a baseline target for this diagnostic index, conventionally normalized to 
\[
R(E)\ge 1,
\]
as a biological viability target, even when mathematical persistence is externally maintained through the prescribed inflow $p(t)$. In the remainder of the paper, we impose this viability condition $R(E)\ge 1$ to ensure that the stock remains intrinsically healthy, and we use it to guide the optimization of discounted harvesting profit under a size-selective threshold policy.

\section{Optimal Control Problem and Pontryagin Maximum Principle}\label{sec:objective}

We derive first-order necessary optimality conditions for the
size-structured harvesting problem. The result is a coupled state--adjoint system
with a pointwise switching rule.

We recall the state equation on $l\in(l_0,l_m)$:
\begin{align}\label{eq:state_pmp}
\begin{aligned}
&\frac{\partial x(t,l)}{\partial t}
+\frac{\partial}{\partial l}\Big(g(E(t),l)\,x(t,l)\Big) \\[4pt]
={}&
-\mu(E(t),l)\,x(t,l)-u(t,l)\,x(t,l),
\qquad t>0,\; l\in(l_0,l_m),
\end{aligned}
\end{align}
with boundary inflow $g(E(t),l_0)\,x(t,l_0)=p(t)$ and crowding index
$E(t)=\int_{l_0}^{l_m}\chi(l)\,x(t,l)\,dl$.

Let $c(l)\ge 0$ denote the market value of a fish of length $l$. We consider the
infinite-horizon discounted revenue function
\begin{equation}\label{eq:objective_J}
J[u]
\coloneqq
\int_{0}^{\infty} e^{-rt}
\int_{l_0}^{l_m} c(l)\,u(t,l)\,x(t,l)\,dl\,dt,
\qquad r>0,
\end{equation}
where $r>0$ is the economic discount rate. The term $u(t,l)\,x(t,l)$ represents
the harvest flow from size class $l$ at time $t$.

The harvesting decision is constrained by physical capacity, regulatory limits,
and enforcement:
\begin{equation}\label{eq:control_bounds}
0\le u(t,l)\le u_{\max},
\qquad (t,l)\in[0,\infty)\times[l_0,l_m],
\end{equation}
where $u_{\max}>0$ is the maximum admissible harvest intensity.
Therefore, we define the admissible control set by
\begin{equation}\label{eq:admissible_set}
\mathcal{U}_{\mathrm{ad}}
\coloneqq
\left\{
u:[0,\infty)\times[l_0,l_m]\to[0,u_{\max}]
\;:\;
u \text{ is measurable}
\right\}.
\end{equation}
For every $u\in\mathcal{U}_{\mathrm{ad}}$, we assume that the state equation
\eqref{eq:fish_model} admits a unique nonnegative solution
\[
x^u \in C\big([0,\infty);L^1([l_0,l_m])\big),
\]
under the standing assumptions of Assumption~\ref{ass:standing}. Thus the control
problem consists of maximizing $J[u]$ over all admissible controls
$u\in\mathcal{U}_{\mathrm{ad}}$.

To facilitate a quick understanding of the maximum principle and maintain the coherence of the text, we have reduced the mathematical rigor in the main body (Propositions~\ref{prop:necessary} and~\ref{prop:weak}), retaining only the formal mathematical expressions and corresponding proofs. Readers with a strong mathematical requirement can refer to the rigorous proofs in the Appendix~\ref{appsec:functional}.

\subsection{Lagrangian and variational framework}\label{subsec:lagrangian}

The costate $\lambda(t,l)$ in the adjoint equation satisfies a formal backward transport equation
derived in the following proposition.

\begin{proposition}[Formal first-order necessary conditions]
\label{prop:necessary}
Assume that the coefficients, controls, and state variables satisfy
Assumption~\ref{ass:standing} and are sufficiently smooth for integration by parts to
be justified formally. Let $(\hat u, \hat x, \hat E)$ be an optimal solution to the problem
\eqref{eq:objective_J} and \eqref{eq:state_pmp} subject to
$g(\hat E(t),l_0)\,\hat x(t,l_0)=p(t)$. 
Then, formally, there exists an adjoint variable
$\lambda:[0,\infty)\times[l_0,l_m]\to\mathbb{R}$ satisfying:
\begin{align}\label{eq:adjoint}
\partial_t \lambda(t,l) &+ g(\hat E(t),l)\,\partial_l \lambda(t,l) \notag\\
&= \bigl(r + \mu(\hat E(t),l) + \hat u(t,l)\bigr)\lambda(t,l) - c(l)\,\hat u(t,l)
\notag \\
&\quad + \chi(l)\!\int_{l_0}^{l_m} \hat x(t,s) \Big(
   \lambda(t,s)\,\partial_E\mu(\hat E(t),s)
   - \partial_s\lambda(t,s)\,\partial_Eg(\hat E(t),s) \Big)\,ds.
\end{align}

The terminal-size and transversality conditions are:
\begin{align}
\lambda(t,l_m) &= 0 \qquad \text{for } t \ge 0,
   \label{eq:adjoint_boundary} \\
\lim_{T\to\infty} \int_{l_0}^{l_m} e^{-rT}\lambda(T,l)\,\delta x(T,l)\,dl &= 0
   \qquad \text{for every admissible variation } \delta x.
   \label{eq:transversality}
\end{align}

Define the switching function
\begin{equation}\label{eq:switching}
\Phi(t,l) \coloneqq \hat x(t,l)\bigl( c(l) - \lambda(t,l) \bigr).
\end{equation}
Assuming $\hat x(t,l)>0$ for a.e.\ $(t,l)$, the optimal harvesting profile satisfies the
bang--bang rule:
\begin{equation}\label{eq:opt_control}
\hat u(t,l)=
\begin{cases}
0, & c(l)<\lambda(t,l),\\[4pt]
u_{\max}, & c(l)>\lambda(t,l),\\[4pt]
\text{any value in }[0,u_{\max}], & c(l)=\lambda(t,l).
\end{cases}
\end{equation}
\end{proposition}

\begin{proof}
Define the Lagrangian by adjoining \eqref{eq:state_pmp} to the objective
\eqref{eq:objective_J} via the multiplier $\lambda(t,l)$:
\begin{align*}
L(x, u, \lambda) 
&= \int_0^\infty e^{-rt} \int_{l_0}^{l_m} c(l)u(t,l)x(t,l)\,dl\,dt \\
&\quad - \int_0^\infty e^{-rt} \int_{l_0}^{l_m} \lambda(t,l) \Big[ \partial_t x(t,l) + \partial_l\bigl(g(E(t),l)x(t,l)\bigr) \\
&\hspace{4.5cm} + \bigl(\mu(E(t),l) + u(t,l)\bigr)x(t,l) \Big]\,dl\,dt.
\end{align*}
Consider admissible perturbations $x \to x + \delta x$ and $u \to u + \delta u$, which induce a variation in the nonlocal environment variable $E \to E + \delta E$ where
\[
\delta E(t) = \int_{l_0}^{l_m} \chi(l)\delta x(t,l)\,dl.
\]
The variation of the Lagrangian with respect to the state, $\delta_x L$, must vanish at the optimum:
\begin{align*}
\delta_x L &= \int_0^\infty e^{-rt} \int_{l_0}^{l_m} \delta x \big(c(l)u\big)\,dl\,dt \\
&\quad - \int_0^\infty e^{-rt} \int_{l_0}^{l_m} \lambda \Big[ \partial_t \delta x + \partial_l(g \delta x + x \partial_Eg \delta E) + (\mu + u)\delta x + x \partial_E\mu \delta E \Big]\,dl\,dt = 0.
\end{align*}
We evaluate the integral terms utilizing integration by parts. For the time derivative, assuming $\delta x(0,l) = 0$ and invoking the transversality condition \eqref{eq:transversality} as $t\to\infty$, we get
\[
-\int_0^\infty e^{-rt} \lambda(t,l) \partial_t \delta x\,dt = \int_0^\infty e^{-rt} \bigl(\partial_t \lambda - r\lambda\bigr)\delta x\,dt.
\]
For the spatial derivatives, integration by parts yields:
\begin{align*}
-\int_{l_0}^{l_m} \lambda \partial_l\big(g \delta x + x \partial_Eg \delta E\big)\,dl 
&= - \Big[ \lambda \big(g \delta x + x \partial_Eg \delta E\big) \Big]_{l_0}^{l_m} 
+ \int_{l_0}^{l_m} \partial_l \lambda \big(g \delta x + x \partial_Eg \delta E\big)\,dl.
\end{align*}
Under the terminal-size condition $\lambda(t,l_m) = 0$, the boundary term at $l_m$ vanishes. 
At the recruitment boundary $l_0$, the admissible variations must preserve the inflow flux condition
\[
g(E(t),l_0)x(t,l_0)=p(t),
\]
with $p(t)$ exogenous. Therefore,
\[
\delta\!\big(g(E(t),l_0)x(t,l_0)\big)=0,
\]
that is,
\[
g(E(t),l_0)\,\delta x(t,l_0)+x(t,l_0)\,\partial_E g(E(t),l_0)\,\delta E(t)=0.
\]
Hence the boundary contribution at $l_0$ vanishes:
\[
-\Big[\lambda(t,l)\big(g\,\delta x+x\,\partial_E g\,\delta E\big)\Big]_{l=l_0}=0.
\]

Isolating all terms proportional to $\delta E(t)$ from the spatial and mortality variations, we define the environment coupling factor $C(t)$:
\[
C(t) = \int_{l_0}^{l_m} x(t,s)\Big[ \partial_s\lambda(t,s)\partial_Eg(E(t),s) - \lambda(t,s)\partial_E\mu(E(t),s) \Big]\,ds.
\]
Substituting $\delta E(t) = \int_{l_0}^{l_m} \chi(l)\delta x(t,l)\,dl$ back into the variational equation, this contributes $\chi(l)C(t)$ inside the spatial integral with respect to $\delta x(t,l)$.

Collecting all coefficients for $\delta x(t,l)$ inside the double integral gives:
\begin{align*}
\delta_x L &= \int_0^\infty e^{-rt} \int_{l_0}^{l_m} \delta x(t,l) \Bigg[ c(l)u + \partial_t \lambda - r\lambda + g\partial_l\lambda - \lambda(\mu + u) + \chi(l)C(t) \Bigg]\,dl\,dt = 0.
\end{align*}
Since this must hold for any arbitrary variation $\delta x$, the expression in the square brackets must vanish a.e. Rearranging this yields the adjoint equation \eqref{eq:adjoint}.

Finally, taking the variation with respect to the control $u$, we obtain:
\[
\delta_u L = \int_0^\infty e^{-rt} \int_{l_0}^{l_m} x(t,l)\big( c(l) - \lambda(t,l) \big)\delta u(t,l)\,dl\,dt.
\]
This defines the switching function $\frac{\partial L}{\partial u}$, dictating the optimal bang-bang policy described in \eqref{eq:opt_control}.
\end{proof}

The adjoint $\lambda(t,l)$ represents the shadow value of an additional marginal fish
of size $l$ at time $t$. In \eqref{eq:adjoint}, the term $c(l)\,u(t,l)$ is a value
source from immediate harvest revenue, while $(r+\mu+u)\lambda$ captures the erosion of
continuation value due to discounting and total mortality.

\subsection{Weak-Coupling Approximation and Reduced Adjoint System}
\label{subsec:weak}

Having established the formal adjoint equation \eqref{eq:adjoint}, we consider a weak-coupling approximation that leads to the adjoint equation form appearing in the common literature.

\begin{proposition}[Formal weak-coupling approximation of the adjoint equation]
\label{prop:weak}
Assume that, in the adjoint equation \eqref{eq:adjoint}, the nonlocal coupling term
arising from the dependence of $g(E,l)$ and $\mu(E,l)$ on the crowding variable $E$
is neglected as a formal weak-coupling approximation. Equivalently, the forward state
dynamics retain their full density dependence through $g(E,l)$ and $\mu(E,l)$, while
the adjoint is approximated by dropping the terms involving $\partial_E g$ and
$\partial_E\mu$.

Then the adjoint equation is formally approximated by
\begin{equation}\label{eq:adjoint_reduced}
-\partial_t\lambda(t,l)-g(E(t),l)\,\partial_l\lambda(t,l)
=
c(l)\,u(t,l)-\bigl(r+\mu(E(t),l)+u(t,l)\bigr)\lambda(t,l).
\end{equation}
\end{proposition}

Since $\hat x(t,l)>0$ for a.e.\ $(t,l)$, the switching function
\[
\Phi(t,l)=\hat x(t,l)\bigl(c(l)-\lambda(t,l)\bigr)
\]
defined in \eqref{eq:switching} has the same sign as $c(l)-\lambda(t,l)$. 
It is
therefore convenient to introduce the reduced switching function
\begin{equation}\label{eq:switchfunc_def}
S(t,l)\coloneqq c(l)-\lambda(t,l).
\end{equation}
Under the reduced adjoint equation \eqref{eq:adjoint_reduced}, the optimal control
rule \eqref{eq:opt_control} can be written compactly as
\begin{equation}\label{eq:u_star_sign}
\hat u(t,l)=
\begin{cases}
0, & S(t,l)<0,\\[4pt]
u_{\max}, & S(t,l)>0.
\end{cases}
\end{equation}
At points where $S(t,l)=0$, any value in $[0,u_{\max}]$ is formally admissible.

The weak-coupling approximation is useful in computation because the reduced adjoint
equation \eqref{eq:adjoint_reduced} removes the nonlocal integral term from the
backward equation. The approximation is expected to be reasonable when the
density sensitivities $\partial_E g$ and $\partial_E\mu$ are small relative to the
leading transport, mortality, and discount terms. However, it remains a formal
simplification rather than a rigorously justified asymptotic reduction unless it is
supplemented by a quantitative error estimate, which is beyond the scope of the
present paper.

We now derive explicit error bounds for the weak-coupling approximation using concrete function forms for $g$ and $\mu$ in Section~\ref{sec:numerical}.
\begin{proposition}[Explicit first-order bound]\label{prop:wcbound}
With $g(E,l;\varepsilon)=K(L_\infty-l)/(1+\varepsilon\alpha E)$,
$\mu(E,l;\varepsilon)=\mu_0+\varepsilon\mu_1E$ (so $\sup|\partial_E g|\le\kappa_g\varepsilon$,
$\sup|\partial_E\mu|\le\kappa_\mu\varepsilon$), let $\lambda,\tilde\lambda$ solve the
full \eqref{eq:adjoint} and reduced \eqref{eq:adjoint_reduced} adjoints along a
common trajectory. Set $\Lambda_0=\|\lambda\|_\infty$,
$\Lambda_1=\|\partial_l\lambda\|_\infty$, $M_1=\sup_t\|x\|_{L^1}$,
$\beta=r+\bar\mu+u_{\max}$, $A=\|\chi\|_\infty M_1(\kappa_g\Lambda_1+\kappa_\mu\Lambda_0)$.
Then, with $\tau$ the size-transit time,
\[
\|\lambda-\tilde\lambda\|_{L^\infty(\Qt)}\le\frac{A}{\beta}(e^{\beta\tau}-1)\varepsilon=:C_T\varepsilon,
\qquad \|S-\tilde S\|_\infty=\|\lambda-\tilde\lambda\|_\infty .
\]
\end{proposition}
\begin{proof}
$e=\lambda-\tilde\lambda$ obeys, with $e=0$ at $l_m$ and $t=T$,
$-\partial_t e-g\partial_l e=-(r+\mu+u)e+\chi\,\mathcal C_\varepsilon[\lambda]$,
$|\mathcal C_\varepsilon[\lambda]|\le M_1(\kappa_g\Lambda_1+\kappa_\mu\Lambda_0)\varepsilon$.
Integration along backward characteristics (span $\le\tau$) and Gr\"onwall give
$|e|\le(A\varepsilon/\beta)(e^{\beta\tau}-1)$.
\end{proof}
\begin{remark}[Effective constant and valid range]\label{rem:validrange}
$C_T$ is worst-case; cancellation in $\mathcal C_\varepsilon$ makes the effective
constant $\Ord(1)$, matching the observed ratio $\approx1.5$
(Table~\ref{tab:weak_coupling_error}). The reduction is first-order accurate when
$\varepsilon\alpha E\ll1$ (growth) and $\varepsilon\mu_1E\ll\mu_0+r$ (mortality).
For the calibration $\alpha E\approx0.5$ and $\mu_1E/(\mu_0+r)\approx0.04$: the
mortality coupling is weak, the growth coupling moderate, and the error is
dominated by $\partial_E g$.
\end{remark}

\subsection{Threshold Structure of the Optimal Harvest Policy}
\label{subsec:threshold}

\begin{assumption}[Single-crossing assumptions for the switching function]
\label{ass:switching}
Fix $t\ge 0$. Assume that:
\begin{enumerate}[label=(\roman*)]
\item $c(l)$ is continuous and strictly increasing on $[l_0,l_m]$;
\item $\lambda(t,l)$ is continuous on $[l_0,l_m]$;
\item\label{ass:lsc} Fix $t$. $c\in C^1$ with $c'>0$; $\lambda(t,\cdot)\in C^1$; at every zero $\ell$ of
$S(t,\cdot)$, $\ \partial_l S(t,\ell)=c'(\ell)-\partial_l\lambda(t,\ell)>0$; and
$S(t,\cdot)\not\equiv0$ on any subinterval.
\end{enumerate}
\end{assumption}

Under Assumption~\ref{ass:switching}, the reduced switching function
\[
S(t,l)=c(l)-\lambda(t,l)
\]
is continuous and strictly increasing in $l$ for each fixed $t$. Condition
(i) reflects the natural assumption that larger fish command a higher market
value, while condition (iii) imposes a single-crossing structure by requiring
the switching function $S(t,l)$ to be increasing in size.

\begin{lemma}[Unique threshold]\label{lem:cross}
Under Assumption~\ref{ass:switching}, $S(t,\cdot)$ has at most one zero. If
$S(t,l_0)<0<S(t,l_m)$, a unique threshold $E(t)\in(l_0,l_m)$ exists with $S<0$
on $(l_0,E(t))$, $S>0$ on $(E(t),l_m)$; hence $\hat u=0$ for $l<E(t)$, $\hat u=u_{\max}$
for $l>E(t)$, and $t\mapsto \hat E(t)$ is measurable.
\end{lemma}
\begin{proof}
Every zero is a strict upward crossing, so two zeros would require an intervening
downward crossing, which is impossible. Existence is the intermediate value theorem; the
policy is Theorem~\ref{thm:pmp}; measurability follows from continuity in $l$.
\end{proof}

The threshold policy (Theorem~\ref{thm:pmp}) has a direct operational
interpretation: fish below the threshold size $E(t)$ are protected, while fish
above $\hat E(t)$ are harvested at the maximum admissible intensity.
In practice, this corresponds to a minimum-size harvesting rule that can be
implemented through minimum landing size regulations, mesh-size restrictions and
gear selectivity designed to reduce the capture of fish below $ E(t)$.

If the optimal regime converges to a steady state, so that both the state and
adjoint become time-independent, then $E(t)$ becomes a stationary threshold value $E^*$ and the optimal policy
reduces to the stationary threshold rule
\begin{equation}\label{eq:threshold_policy_stationary}
\hat u(l)=
\begin{cases}
0, & l<E^*,\\[4pt]
u_{\max}, & l>E^*.
\end{cases}
\end{equation}

Under long-run convergence, the stationary threshold $E^*$ is characterized by
the steady-state counterpart of the adjoint equation together with the closure
relation \eqref{eq:F_equation}. This connects the optimal-control analysis in this section to the steady-state framework of Section~\ref{sec:steady_state}, and provides a consistent description of the
long-run optimal harvesting regime.

\section{Numerical Simulation and Case Study}\label{sec:numerical}

This section illustrates the analytical results of the preceding sections
through a numerical case study calibrated to a managed gadoid fishery. The
numerical setup is chosen to be broadly consistent with Atlantic cod
(\textit{Gadus morhua}), a stock for which size-based management measures such as
minimum fish sizes and annual catch limits are standard practice \cite{tallack_regional_2009,noaa_fisheries_atlantic_2026,liang2026separation}.

\subsection{Parameter specification and function forms}\label{subsec:params}

For the numerical case study, we take the entry size to be $l_0=20$\,cm and the upper boundary size to be $l_m=130$\,cm, so that $l\in[20,130]$\,cm. To preserve the standing assumption $g(E,l)\ge g_{\min}>0$ on the modeled interval, the asymptotic von~Bertalanffy length is chosen strictly above the computational upper bound,
\[
L_\infty = 135.3\ \text{cm} > l_m=130\ \text{cm},
\]
which avoids the degeneracy that would occur if $l_m=L_\infty$.

Individual growth follows the density-dependent von~Bertalanffy form
\begin{equation}\label{eq:gEL_num}
g(E,l)=\frac{K\,(L_\infty-l)}{1+\alpha E},
\end{equation}
with
\[
L_\infty=135.3\ \text{cm},\qquad
K=0.17\ \mathrm{yr}^{-1},\qquad
\alpha=5\times 10^{-6}\ \mathrm{individual}^{-1}.
\]
The values $L_\infty=135.3$ cm and $K=0.17\,\mathrm{yr}^{-1}$ are consistent with a regional Atlantic cod von Bertalanffy calibration reported by Tallack et al. \cite{tallack_regional_2009,noaa_fisheries_atlantic_2026,liu2026computational}.

Natural mortality is modeled as a linearly density-dependent rate,
\begin{equation}\label{eq:muEL_num}
\mu(E,l)=\mu_0+\mu_1 E,
\end{equation}
with
\[
\mu_0=0.20\ \mathrm{yr}^{-1},\qquad
\mu_1=1\times 10^{-7}\ \mathrm{yr}^{-1}\,\mathrm{individual}^{-1}.
\]
The baseline value $\mu_0=0.20\,\mathrm{yr}^{-1}$ is adopted as a plausible adult-cod benchmark, consistent with DFO guidance noting that instantaneous natural mortality $M\approx 0.2$ is considered normal for adult cod in the 4T--4Vn assessment context \cite{government_of_canada_rebuilding_2025,yu2026pattern}. Moreover, $\partial_E\mu=\mu_1\ge 0$, so the monotonicity requirement of Assumption~\ref{ass:standing} is satisfied.

The crowding index uses a size-squared kernel,
\[
\chi(l)=\chi\,l^2,\qquad \chi=1\times 10^{-4}\ \mathrm{cm}^{-2},
\]
consistent with \eqref{eq:chi_l2}. Since $x(t,l)$ has units of individuals per cm, the crowding index $E(t)$ has units of individuals.

The per-unit harvest value is taken proportional to body weight, which is approximated as scaling with the cube of length:
\begin{equation}\label{eq:cL_num}
c(l)=c_0\,l^3,\qquad c_0=1\times 10^{-5}\ \text{\$\,cm}^{-3}.
\end{equation}
This specification ensures that $c'(l)>0$ on $[l_0,l_m]$.

The size-dependent fertility function used in the intrinsic replacement index \eqref{eq:RE} is specified as
\begin{equation}\label{eq:mL_num}
m(l)=
\begin{cases}
m_0\left(\dfrac{l}{l_m}\right)^3, & l\ge l_{\mathrm{mat}},\\[4pt]
0, & l<l_{\mathrm{mat}},
\end{cases}
\end{equation}
with maturation length
\[
l_{\mathrm{mat}}=50\ \text{cm},
\qquad
m_0=2.0\ \mathrm{yr}^{-1}.
\]
Only individuals above $l_{\mathrm{mat}}$ contribute to the replacement index, and the value $l_{\mathrm{mat}}=50$\,cm should be interpreted as a plausible gadoid-fishery calibration rather than a universal cod maturity threshold.

The external inflow is taken to be
\[
p=5\times 10^4\ \mathrm{individuals}\,\mathrm{yr}^{-1},
\]
representing either hatchery stocking or an effective annual juvenile input at the lower boundary. The discount rate and maximum harvesting intensity are
\[
r=0.05\ \mathrm{yr}^{-1},
\qquad
u_{\max}=0.5\ \mathrm{yr}^{-1},
\]
and all time-dependent simulations are run over a finite horizon
\[
T=60\ \text{yr}.
\]

The parameter values used in the numerical case study are summarized in Table~\ref{tab:params}.

\begin{table}[ht]
\centering
\caption{Parameter values used in the numerical case study.}
\label{tab:params}
\begin{tabular}{llll}
\hline
Symbol & Description & Value & Unit\\
\hline
$l_0$ & entry size & 20 & cm\\
$l_m$ & upper boundary size & 130 & cm\\
$L_\infty$ & asymptotic length & 135.3 & cm\\
$K$ & Brody growth coefficient & 0.17 & $\mathrm{yr}^{-1}$\\
$\alpha$ & growth crowding sensitivity & $5\times 10^{-6}$ & $\mathrm{individual}^{-1}$\\
$\mu_0$ & baseline natural mortality & 0.20 & $\mathrm{yr}^{-1}$\\
$\mu_1$ & crowding mortality sensitivity & $1\times 10^{-7}$ & $\mathrm{yr}^{-1}\,\mathrm{individual}^{-1}$\\
$\chi$ & crowding kernel coefficient & $1\times 10^{-4}$ & $\mathrm{cm}^{-2}$\\
$c_0$ & price coefficient & $1\times 10^{-5}$ & $\$\,\mathrm{cm}^{-3}$\\
$m_0$ & fertility scale & 2.0 & $\mathrm{yr}^{-1}$\\
$l_{\mathrm{mat}}$ & maturation length & 50 & cm\\
$p$ & external inflow flux & $5\times 10^{4}$ & $\mathrm{individuals}\,\mathrm{yr}^{-1}$\\
$r$ & discount rate & 0.05 & $\mathrm{yr}^{-1}$\\
$u_{\max}$ & maximum harvest intensity & 0.50 & $\mathrm{yr}^{-1}$\\
$T$ & simulation horizon & 60 & yr\\
\hline
\end{tabular}
\end{table}

\subsection{Numerical methods}\label{subsec:methods}

The calibrated functions above are implemented numerically on a uniform size grid with $N_l=400$ cells and mesh width
\[
\Delta l=\frac{l_m-l_0}{N_l}=0.275\ \text{cm}.
\]
Time is advanced with a constant step
\[
\Delta t = 0.0112239\ \text{yr},
\]
computed from a Courant--Friedrichs--Lewy (CFL) restriction with safety factor $0.8$ using the maximum growth rate over the computational domain and coefficient range.

The transport equation \eqref{eq:fish_model} is solved by a first-order upwind finite-volume scheme. At each time step, the crowding index $E(t)$ is evaluated numerically from \eqref{eq:Edef}, the coefficients $g(E,l)$ and $\mu(E,l)$ are updated, and the lower-boundary inflow is imposed through
\[
x(t,l_0)=\frac{p(t)}{g(E(t),l_0)}.
\]

The closure equation \eqref{eq:F_equation} is solved with Brent's bracketing method, and the no-harvest stationary profile is then evaluated from \eqref{eq:ss_profile_noharvest}. As a consistency check, the time-dependent PDE is also integrated forward with $u\equiv 0$ until the numerical solution stabilizes.

The intrinsic replacement index \eqref{eq:RE} is evaluated by composite Simpson quadrature. Because the simulations are run over a finite horizon $T=60$\,yr, the numerical objective is the truncated discounted revenue
\begin{equation}\label{eq:JT_num}
J_T[u]\coloneqq
\int_0^T e^{-rt}\int_{l_0}^{l_m} c(l)\,u(t,l)\,x(t,l)\,dl\,dt,
\end{equation}
which serves as a numerical approximation to the infinite-horizon objective \eqref{eq:objective_J}.

For each candidate threshold $E\in[l_0,l_m]$, the PDE is solved under the stationary threshold rule \eqref{eq:threshold_policy_stationary}, and the corresponding truncated discounted revenue $J_T(E)$ at the threshold is computed. The resulting steady-state crowding, total population, and replacement index are then used to compare the biological and economic implications of alternative threshold policies.

\begin{proposition}[Truncation bound]\label{prop:trunc}
With $\mu\ge\mu_{\min}>0$ (hence $\|x(t)\|_{L^1}\le p/\mu_{\min}$),
\[
0\le J_\infty-J_T\le\frac{e^{-rT}}{r}\bar\rho,\qquad
\bar\rho=c_0 l_m^3 u_{\max}\frac{p}{\mu_{\min}} .
\]
If the controlled state converges to a steady regime with revenue rate
$\rho_\infty$, then $\widehat J_T=J_T+e^{-rT}\rho_\infty/r$ satisfies
$|J_\infty-\widehat J_T|=\Ord(e^{-rT}\sup_{t\ge T}|\rho(t)-\rho_\infty|)$.
\end{proposition}

\subsection{Results}\label{subsec:results}

Solving the closure equation \eqref{eq:F_equation} gives the no-harvest
equilibrium crowding level
\[
E^*=103108.17.
\]
The associated stationary population size is
\[
N^*=\int_{l_0}^{l_m}x^*(l)\,dl = 237004.87.
\]
The corresponding stationary size profile is shown in Figure~\ref{fig:steady_state_profile}.

\begin{figure}[htbp]
\centering
\includegraphics[width=0.8\textwidth]{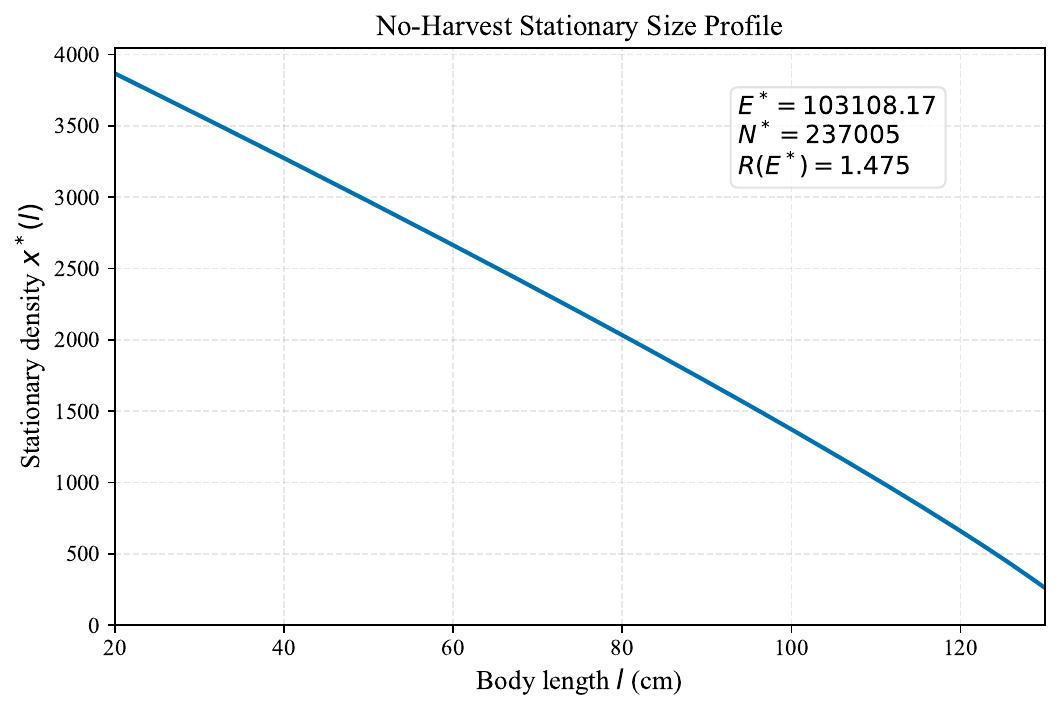}
\caption{No-harvest stationary size profile $x^*(l)$ for the parameter values
in Table~\ref{tab:params}. The figure also reports the corresponding baseline
quantities $E^*=103108.17$, $N^*=237005$, and $R(E^*)=1.475$.}
\label{fig:steady_state_profile}
\end{figure}

Figure~\ref{fig:replacement_index} shows the intrinsic replacement index
$R(E)$ as a function of crowding. The index is strictly decreasing over the
computed range, confirming that stronger crowding lowers intrinsic replacement
capacity. At the no-harvest equilibrium one has $R(E^*)=1.474679>1$, while the
critical value satisfying $R(E)=1$ is $E^*=179008.21$. Since the
boundary inflow is exogenous, this curve is interpreted as a biological
viability diagnostic rather than a persistence threshold.

\begin{figure}[htbp]
\centering
\includegraphics[width=0.8\textwidth]{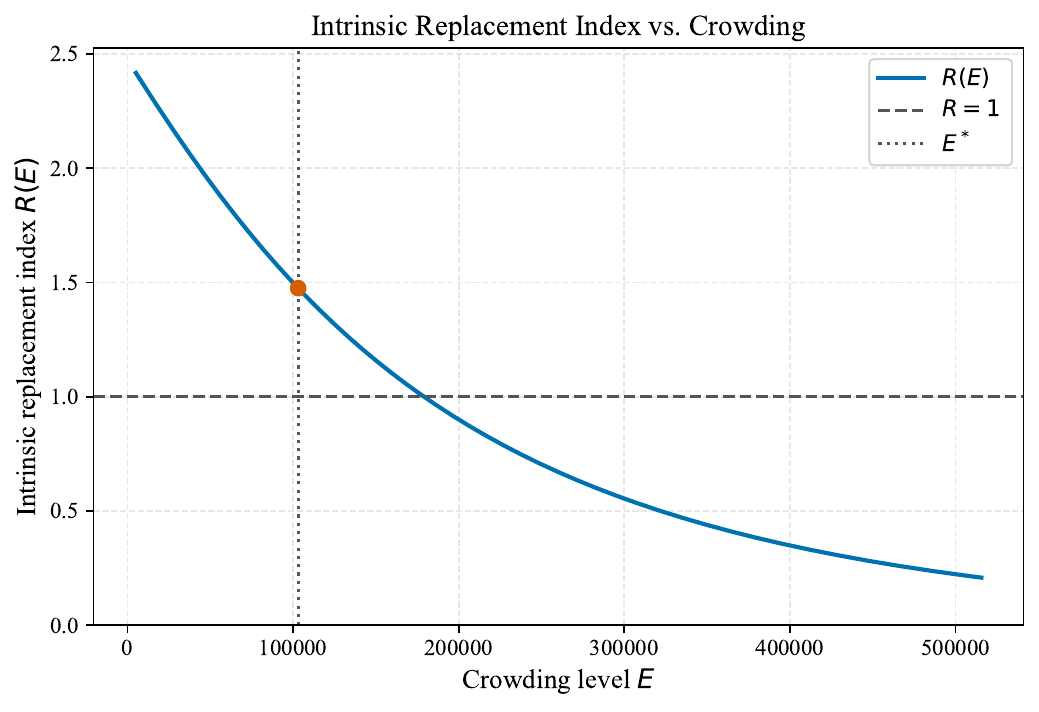}
\caption{Intrinsic replacement index $R(E)$ as a function of the crowding
level. The dashed horizontal line marks $R=1$, and the dashed vertical line
marks the no-harvest equilibrium $E^*=103108.17$.}
\label{fig:replacement_index}
\end{figure}

The full PDE \eqref{eq:fish_model} is solved forward in time under three
representative threshold policies,
\[
 l=40\ \text{cm},\qquad l=60\ \text{cm},\qquad l=80\ \text{cm}.
\]
The resulting trajectories of $E(t)$ and $N(t)$ are shown in
Figure~\ref{fig:time_dynamics}. In the present simulations, the trajectories of
$E(t)$ and $N(t)$ approach their threshold-dependent stationary regimes within
approximately 8--10 years, based on the time required to enter and remain
within 1\% of their terminal values. Table~\ref{tab:numerical_convergence} lists the numerically observed convergence times for the three representative thresholds. Lower thresholds produce more aggressive harvesting and therefore lower
long-run crowding and population levels.

\begin{figure}[htbp]
\centering
\includegraphics[width=0.8\textwidth]{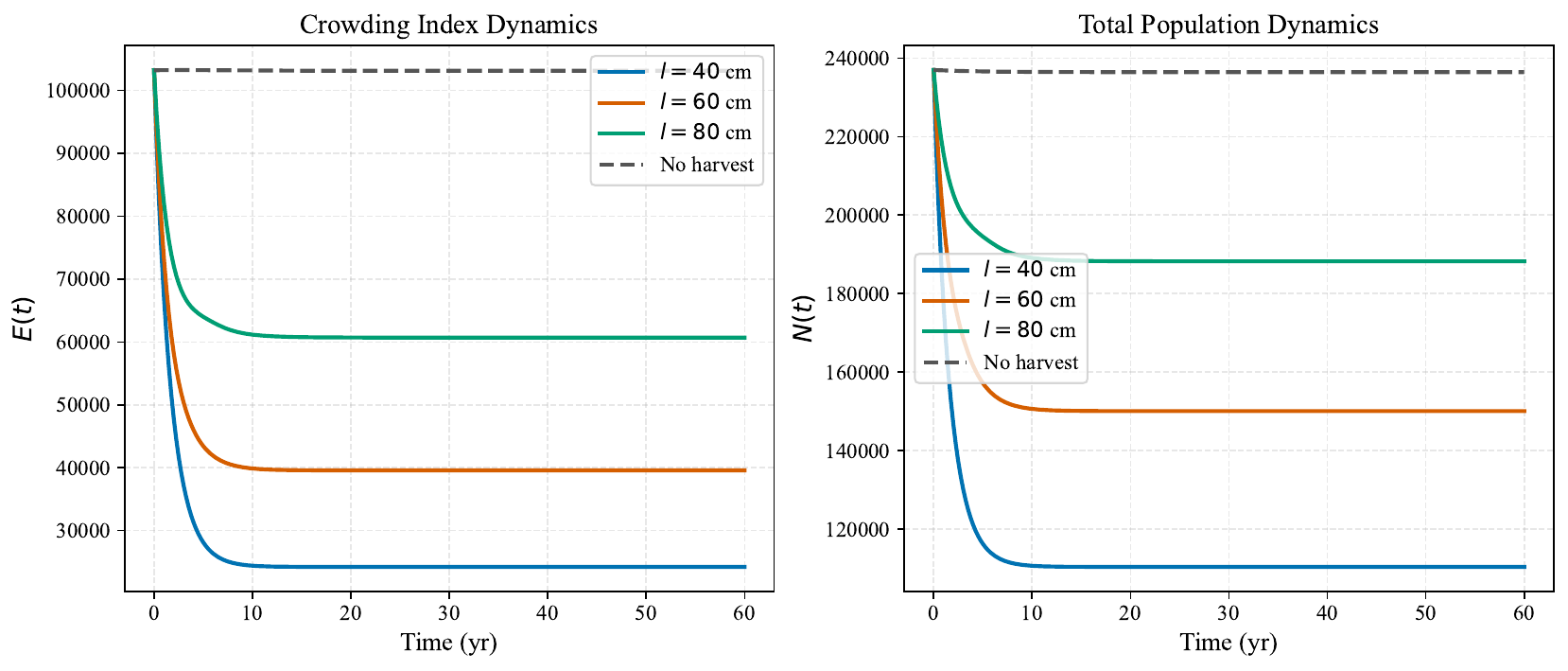}
\caption{Time evolution of the crowding index $E(t)$ and total population
$N(t)$ under representative threshold harvesting policies.}
\label{fig:time_dynamics}
\end{figure}

\begin{table}[htbp]
\centering
\caption{Convergence times for representative thresholds.\label{tab:numerical_convergence}}
\begin{tabular}{cccc}
\hline
Threshold $l$ (cm) & $E^*$ & $N^*$ & 1\% convergence time (yr)\\
\hline
40 & 24229.07 & 110399.50 & $E$: 9.75,\ $N$: 7.76\\
60 & 39572.49 & 150074.01 & $E$: 9.51,\ $N$: 8.26\\
80 & 60697.30 & 188228.22 & $E$: 9.51,\ $N$: 8.26\\
\hline
\end{tabular}
\end{table}

The numerically computed truncated discounted revenue $J_T(l)$ is shown in
Figure~\ref{fig:revenue_threshold}. Over the tested threshold grid, the map
$l\mapsto J_T(l)$ is unimodal and attains its maximum at the interior
threshold
\[
l=66.45\ \text{cm}.
\]
At this threshold, the maximum truncated discounted revenue is
\[
J_T(l)=1.9698\times 10^6.
\]
The corresponding steady state remains biologically viable, with
\[
R(E^*)=1.977621,\qquad
E^*=45967.96,\qquad
N^*=163572.21.
\]
Thus, for the present calibration, the economically optimal threshold also lies
in the region satisfying the viability target $R(E)\ge 1$.

\begin{figure}[ht]
\centering
\includegraphics[width=0.85\textwidth]{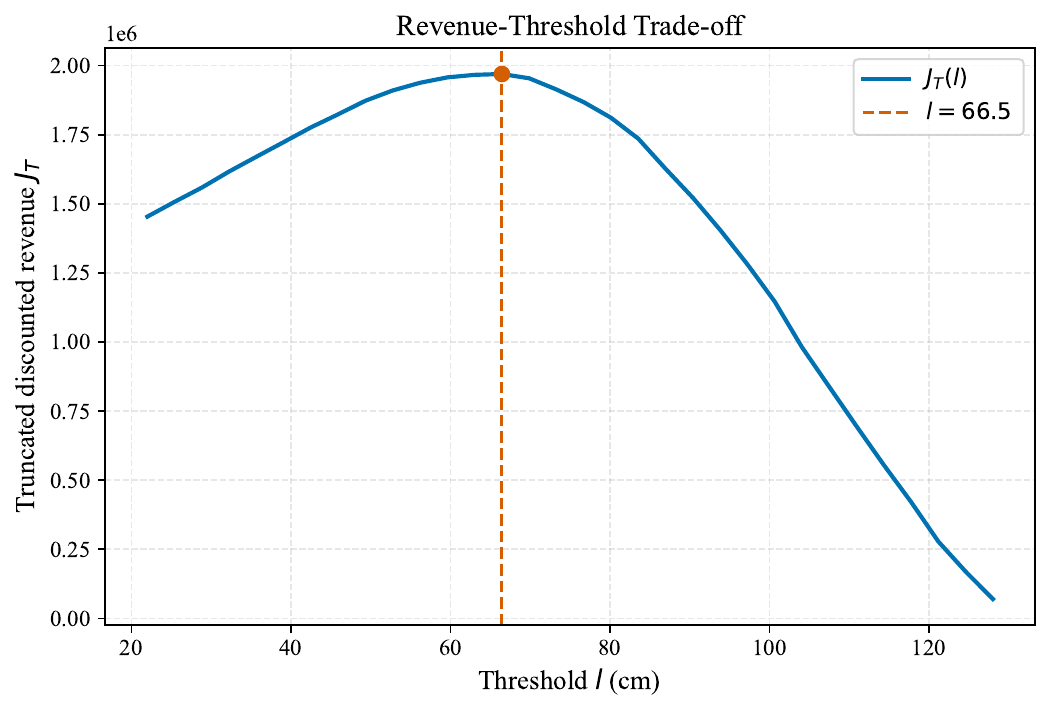}
\caption{Truncated discounted revenue $J_T$ as a function of threshold size
$l$. The vertical dashed line marks the maximizing threshold
$l=66.45$\,cm.}
\label{fig:revenue_threshold}
\end{figure}

Figure~\ref{fig:profile_comparison} compares the no-harvest stationary profile
with the stationary profile associated with the optimal threshold
$\hat l=66.45$\,cm. Relative to the no-harvest baseline, the
optimal policy depletes larger individuals above the threshold while preserving
the smaller size classes below it. The resulting redistribution is consistent
with the threshold structure predicted by the analysis of
Section~\ref{sec:objective}.

\begin{figure}[ht]
\centering
\includegraphics[width=0.85\textwidth]{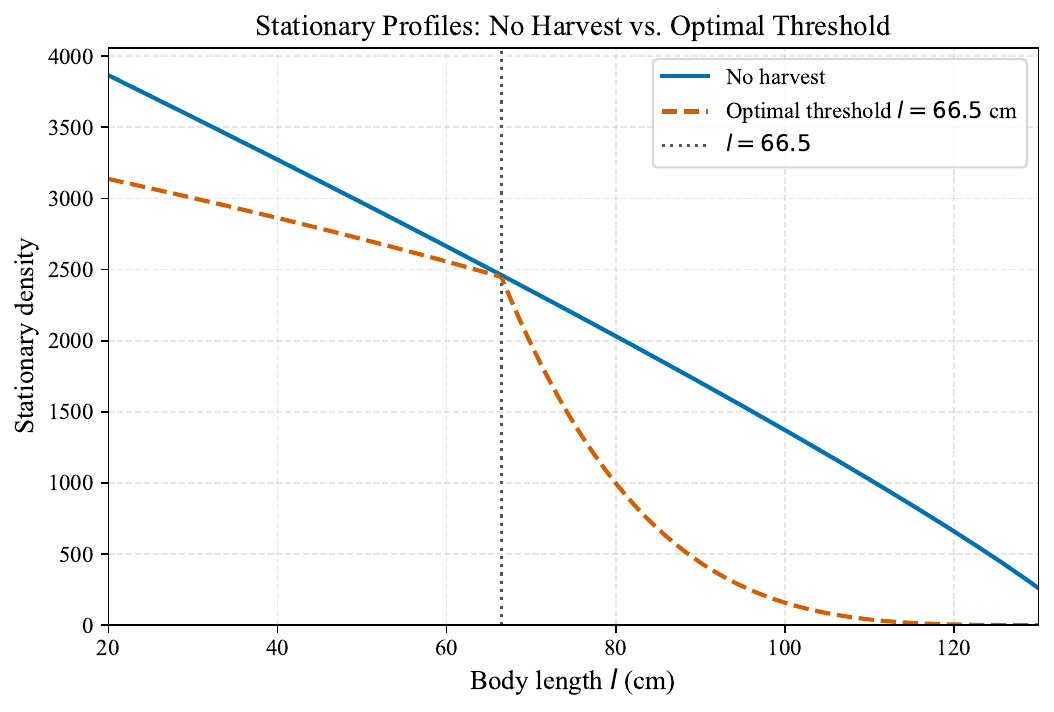}
\caption{Stationary size profiles under no harvesting and under the optimal
threshold policy. The vertical line marks
$E^*=66.45$\,cm.}
\label{fig:profile_comparison}
\end{figure}

To validate the finite-horizon bound and tail correction shown in Proposition~\ref{prop:trunc}, we show the finite-horizon truncation results in Table~\ref{tab:horizon} and Figure~\ref{fig:horizon}. Because the system reaches its steady regime well before $T$, $\widehat J_T$ is essentially horizon-independent. Tail corrected truncation bound provides a more accurate approximation to $J_{\infty}$: the raw $J_T$ at $T=60$ underestimates $J_\infty \approx 2.046 \times 10^6$ by about 4\%; the corrected estimator reduces this to $< 0.1\%$ already at $T=40$. In Figure~\ref{fig:horizon}, the curves for the tail corrected estimator and for $J_\infty$ overlap. 

\begin{table}[t]
\centering
\caption{Finite-horizon truncation at $E=66.45$~cm
($\rho_\infty=8.14\times10^4$). $\widehat J_T$ is essentially horizon-independent;
$J_\infty\approx2.046\times10^6$.}
\label{tab:horizon}
\begin{tabular}{cccc}
\toprule
$T$ (yr) & $J_T$ & $e^{-rT}\rho_\infty/r$ & $\widehat J_T$\\
\midrule
20 & $1.4439\times10^6$ & $5.99\times10^5$ & $2.0430\times10^6$\\
40 & $1.8245\times10^6$ & $2.20\times10^5$ & $2.0449\times10^6$\\
60 & $1.9645\times10^6$ & $8.11\times10^4$ & $2.0456\times10^6$\\
80 & $2.0160\times10^6$ & $2.98\times10^4$ & $2.0459\times10^6$\\
100& $2.0350\times10^6$ & $1.10\times10^4$ & $2.0460\times10^6$\\
200& $2.0459\times10^6$ & $7.4\times10^1$  & $2.0460\times10^6$\\
\bottomrule
\end{tabular}
\end{table}
\begin{figure}[t]\centering
\includegraphics[width=0.85\linewidth]{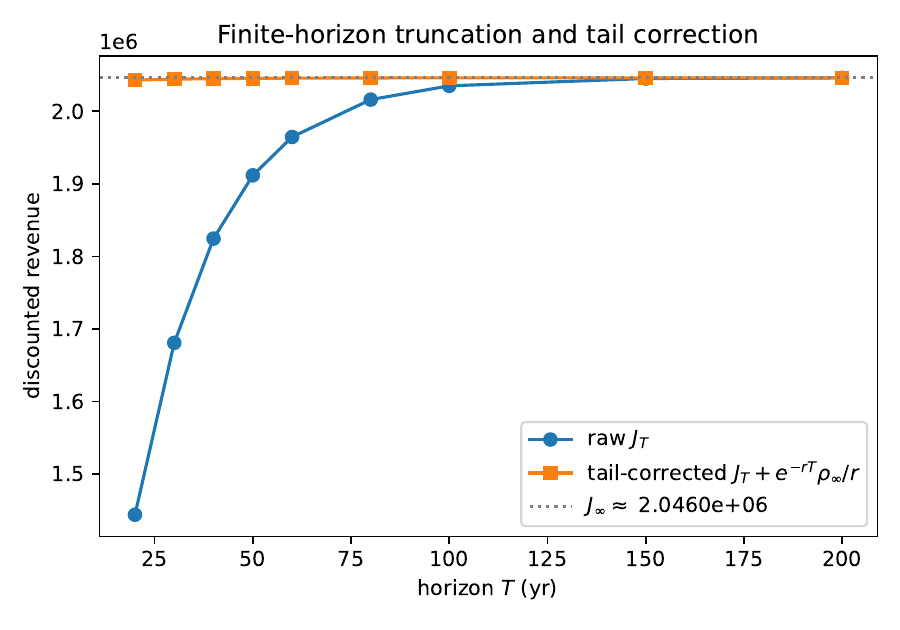}
\caption{Raw $J_T$ vs.\ tail-corrected $\widehat J_T$; the
correction removes $\approx4\%$ of truncation error at $T=60$.}
\label{fig:horizon}
\end{figure}

We conduct convergence tests to show the robustness of the conclusions.
A spatial refinement study at the optimal threshold gives first-order convergence consistent with the upwind scheme (Table~\ref{tab:mesh}; Figure~\ref{fig:mesh}).
\begin{table}[t]\centering
\caption{Spatial mesh convergence ($E^*=66.45$~cm, $T=60$~yr,
CFL$=0.8$). Observed order $p\approx1.0$--$1.3$; production resolution $N_l=400$
within $0.3\%$ of refined. Temporal study (CFL$\in[0.05,0.8]$, $N_l=800$): order
$\approx1$, $J_T$ stable to $<0.05\%$.}
\label{tab:mesh}
\begin{tabular}{ccccc}
\toprule
$N_l$ & $\Delta l$ (cm) & $E^*$ & $N^*$ & $J_T$\\
\midrule
50  & 2.200  & 46{,}090 & 162{,}763 & $2.00491\times10^6$\\
100 & 1.100  & 45{,}892 & 162{,}902 & $1.98232\times10^6$\\
200 & 0.550  & 45{,}792 & 162{,}972 & $1.97096\times10^6$\\
400 & 0.275  & 46{,}022 & 163{,}545 & $1.96454\times10^6$\\
800 & 0.1375 & 45{,}997 & 163{,}563 & $1.96168\times10^6$\\
1600& 0.0688 & 45{,}985 & 163{,}572 & $1.96025\times10^6$\\
3200& 0.0344 & 45{,}943 & 163{,}509 & $1.95964\times10^6$\\
\bottomrule
\end{tabular}
\end{table}

The observed spatial order is $p\approx$ 1.0--1.3; the temporal study (varying CFL at fixed mesh) likewise yields order $\approx$ 1, with $J_T$ stable to $<0.05\%$ across CFL$\in [0.05,0.8]$. The production resolution $(N_l=400,\mathrm{CFL}=0.8)$ is thus within 0.3\% of the refined values.

\begin{figure}[t]\centering
\includegraphics[width=0.85\linewidth]{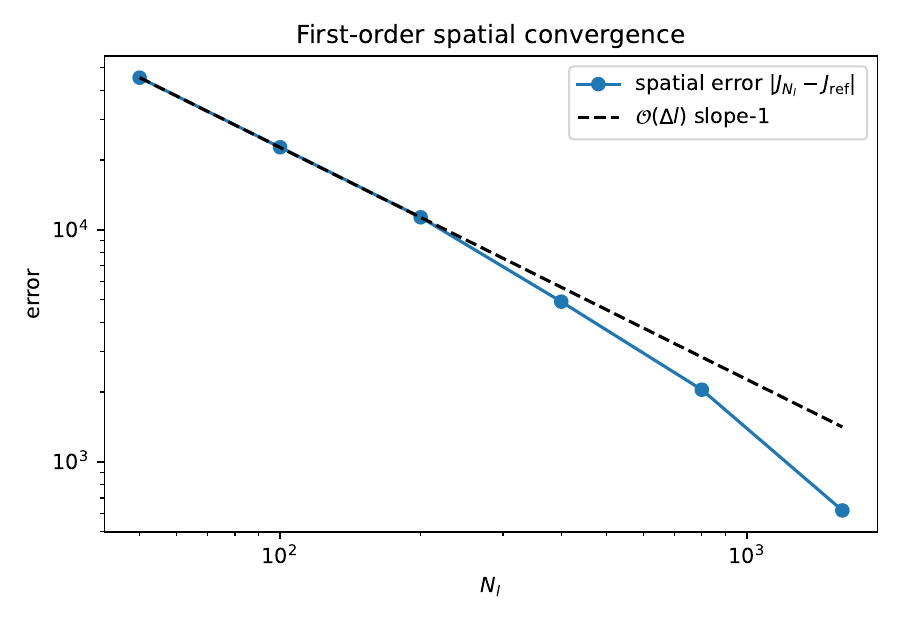}
\caption{Log--log spatial convergence of $J_T$; dashed line is
the first-order reference slope.}
\label{fig:mesh}
\end{figure}

\subsection{Weak-coupling validation}
\label{subsec:weak_validation}

To quantitatively validate the weak-coupling approximation, we introduce an
explicit coupling parameter \(\varepsilon\in(0,1]\) into the density-dependent
terms:
\[
g(E,l;\varepsilon)=\frac{K_{\rm vb}(L_\infty-l)}{1+\varepsilon\alpha E},
\qquad
\mu(E,l;\varepsilon)=\mu_0+\varepsilon\mu_1E .
\]
Then
\[
\partial_E g(E,l;\varepsilon)=\mathcal O(\varepsilon),
\qquad
\partial_E\mu(E,l;\varepsilon)=\mathcal O(\varepsilon),
\]
uniformly on bounded \(E\)-intervals.

Along a fixed admissible state-control trajectory, let \(\lambda\) solve the
full adjoint system and let \(\tilde\lambda\) solve the reduced adjoint system.
The full adjoint may be written in current-value form as
\[
-\partial_t \lambda
-
g(E,l;\varepsilon)\partial_l\lambda
=
c(l)u(l)
-
\bigl(r+\mu(E,l;\varepsilon)+u(l)\bigr)\lambda
+
\chi(l)\mathcal C_\varepsilon[\lambda](t),
\]
where
\[
\mathcal C_\varepsilon[\lambda](t)
=
\int_{l_0}^{l_m}
x(t,s)
\left(
\partial_s\lambda(t,s)\partial_E g(E,s;\varepsilon)
-
\lambda(t,s)\partial_E\mu(E,s;\varepsilon)
\right)\,ds .
\]
The reduced adjoint is obtained by omitting
\(\chi(l)\mathcal C_\varepsilon[\lambda](t)\):
\[
-\partial_t \tilde{\lambda}
-
g(E,l;\varepsilon)\partial_l\tilde{\lambda}
=
c(l)u(l)
-
\bigl(r+\mu(E,l;\varepsilon)+u(l)\bigr)\tilde{\lambda}.
\]

Let \(e=\lambda-\tilde\lambda\). Since
\(\partial_E g=\mathcal O(\varepsilon)\) and \(\partial_E\mu=\mathcal O(\varepsilon)\), the
nonlocal source satisfies
\[
|\mathcal C_\varepsilon[\tilde\lambda](t)|
\le C\varepsilon .
\]
With finite parameter values in Table~\ref{tab:params}, the error
equation for \(e\) gives, by integration along backward characteristics and
Gr\"onwall's inequality,
\[
\|\lambda-\tilde\lambda\|_{L^\infty((0,T)\times(l_0,l_m))}
\le C_T\varepsilon .
\]
Thus the weak-coupling approximation is first-order consistent in the
small-sensitivity regime. The same estimate applies to the switching functions
because \(S-\tilde S=-(\lambda-\tilde\lambda)\).

To test this estimate numerically, we solved the full and reduced adjoint
systems along the same threshold-harvesting trajectory for
\(\varepsilon\in\{0.2,0.1,0.05,0.025,0.0125\}\). Table~\ref{tab:weak_coupling_error}
shows that the maximum adjoint discrepancy decreases approximately linearly
with \(\varepsilon\), and Figure~\ref{fig:weak_coupling_error} confirms the
expected slope-one scaling on a log--log plot.

\begin{table}[htbp]
\centering
\caption{Numerical evaluation of the weak-coupling adjoint discrepancy across decreasing values of the coupling parameter $\varepsilon$, demonstrating the $\mathcal O(\varepsilon)$ empirical convergence rate.}
\label{tab:weak_coupling_error}
\begin{tabular}{ccc}
\hline
\textbf{Coupling Parameter} ($\varepsilon$) & \textbf{Max Adjoint Error} ($\|\lambda - \tilde{\lambda}\|_{\infty}$) & \textbf{Error-to-Epsilon Ratio} ($\|\lambda - \tilde{\lambda}\|_{\infty} / \varepsilon$) \\ \hline
0.2000 & 0.274585 & 1.372924 \\
0.1000 & 0.147568 & 1.475677 \\
0.0500 & 0.076619 & 1.532377 \\
0.0250 & 0.039054 & 1.562159 \\
0.0125 & 0.019718 & 1.577418 \\ \hline
\end{tabular}
\end{table}

\begin{figure}[htbp]
    \centering
    \includegraphics[width=0.9\linewidth]{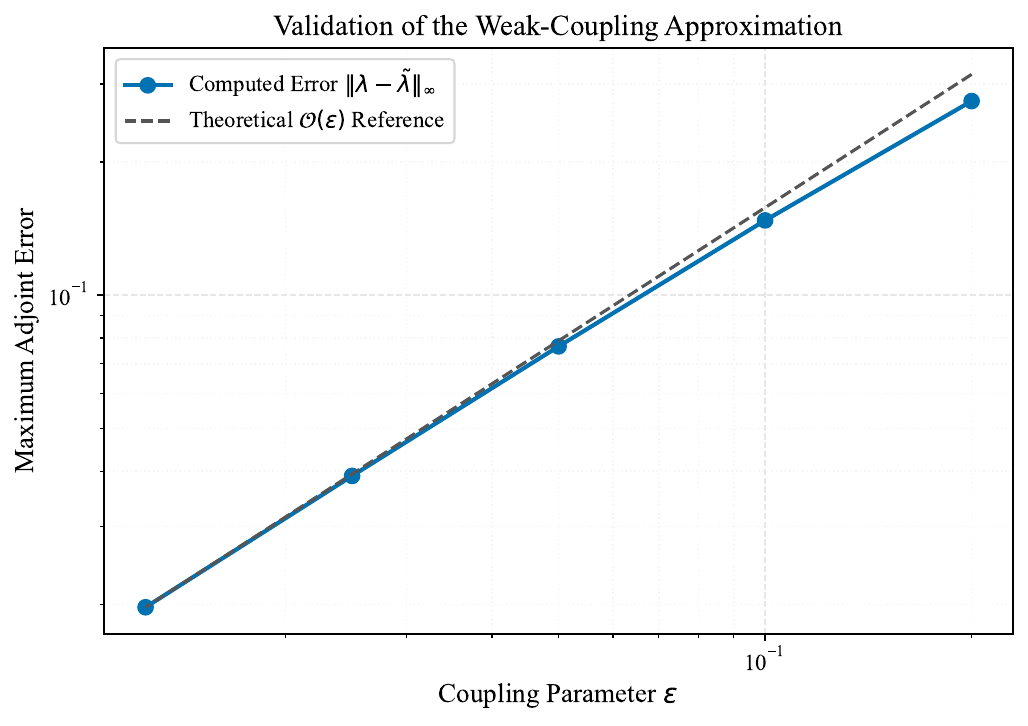}
    \caption{Log-log convergence plot of the maximum absolute discrepancy $\|\lambda - \tilde{\lambda}\|_{\infty}$ between the full non-local adjoint system and the reduced system as a function of the coupling parameter $\varepsilon$. The dashed line represents the theoretical $\mathcal O(\varepsilon)$ reference slope, verifying the strict linear asymptotic bound of the weak-coupling approximation.}
    \label{fig:weak_coupling_error}
\end{figure}

\subsection{Parameter sensitivity analysis}
\label{subsec:sensitivity}

To assess whether the optimal threshold is an artifact of the baseline
calibration, we performed one-at-a-time sensitivity runs for seven parameters:
$\alpha$, $r$, $l_{\rm mat}$, $K$, $\mu_0$, $p$, and $u_{\max}$. Specifically, we varied
\[
\alpha \in \{0,\;2.5\times 10^{-6},\;5\times 10^{-6},\;7.5\times 10^{-6},\;10^{-5}\},
\]
\[
r\in\{0.02,\;0.05,\;0.08,\;0.10\},
\]
\[
l_{\mathrm{mat}}\in\{40,\;50,\;60,\;70\}\ {\rm cm},
\]
\[
K\in \{0.10,0.15,0.17,0.20,0.25\}\ {\rm yr}^{-1} 
\]
\[
\mu_0\in \{0.10,0.15,0.20,0.25,0.30\}\ {\rm yr}^{-1},
\]
\[
p\in \{2.5\times 10^4,5.0\times 10^4,7.5\times 10^4,1.0\times 10^5\}\ {\rm ind.yr}^{-1},
\]
and
\[
u_{\max} \in \{0.25,0.50,0.75,1.00\} \ {\rm yr}^{-1}.
\]
For each parameter value, all other parameters were held at their baseline
values from Table~\ref{tab:params}. We recomputed the threshold-dependent
finite-horizon objective \(J_T(l)\), the maximizing threshold
\(l\), the associated steady crowding level \(E^*\), and
the intrinsic replacement index \(R(E^*)\).

The sensitivity runs are summarized in Table~\ref{tab:sensitivity} for all seven parameters and visually depicted in Figure~\ref{fig:sensitivity_summary} for $\alpha$, $r$ and $l_{\rm mat}$. 
Across the tested ranges, the revenue-maximizing threshold strictly remains interior, confirming that extreme bang-bang policies (e.g., zero harvest or harvesting at the minimum possible size) are suboptimal. As shown in Figure~\ref{fig:sensitivity_summary}(a), increasing the crowding sensitivity \(\alpha\) shifts the optimal threshold downward. This occurs because stronger density dependence significantly suppresses individual growth rates, increasing the time-discounting penalty of waiting for fish to grow larger, thus making earlier harvesting more economically favorable. Similarly, increasing the discount rate \(r\) shifts the threshold downward (Figure~\ref{fig:sensitivity_summary}b), reflecting the greater economic weight placed on immediate harvest yields over delayed future revenues. Finally, Figure~\ref{fig:sensitivity_summary}(c) illustrates that increasing the maturation length \(l_{\mathrm{mat}}\) monotonically lowers the intrinsic replacement index, thereby narrowing the set of thresholds capable of satisfying the biological viability target \(R(E^*)\ge 1\). 

\begin{table}[htbp]
\centering
\caption{Expanded one-at-a-time sensitivity (others at baseline).
Interior optimum robust; at $\mu_0=0.30$, $\R(E^*)<1$.}
\label{tab:sensitivity}
\begin{tabular}{llcccc}
\hline
Parameter varied & Value & \(l\) (cm) & \(J_T(l)\) & \(E^*\) & \(R(E^*)\)\\
\toprule
\(\alpha\) & \(0\) & 76.71 & \(2.68\times10^6\) & 59{,}001 & 2.351 \\
  & \(2.5\times10^{-6}\) & 69.87 & \(2.260\times10^6\) & 50{,}256 & 2.144 \\
  & \(5.0\times10^{-6}\) & 66.45 & \(1.970\times10^6\) & 45{,}968 & 1.978 \\
  & \(7.5\times10^{-6}\) & 59.61 & \(1.760\times10^6\) & 38{,}749 & 1.896 \\
 & \(1.0\times10^{-5}\) & 56.19 & \(1.602\times10^6\) & 35{,}259 & 1.809 \\
\midrule
\(r\) & \(0.02\) & 66.45 & \(3.297\times10^6\) & 45{,}967 & 1.978 \\
  & \(0.05\) & 66.45 & \(1.970\times10^6\) & 45{,}967 & 1.978 \\
  & \(0.08\) & 59.61 & \(1.416\times10^6\) & 39{,}272 & 2.045 \\
  & \(0.10\) & 59.61 & \(1.213\times10^6\) & 39{,}272 & 2.045 \\
\midrule
\(l_{\mathrm{mat}}\) (cm) & \(40\) & 66.45 & \(1.970\times10^6\) & 45{,}967 & 2.025 \\
                & \(50\) & 66.45 & \(1.970\times10^6\) & 45{,}967 & 1.978 \\
                & \(60\) & 66.45 & \(1.970\times10^6\) & 45{,}967 & 1.898 \\
                & \(70\) & 66.45 & \(1.970\times10^6\) & 45{,}967 & 1.774 \\
\midrule
$K$ (yr$^{-1}$) & 0.12 & 56.75 & $1.318\times10^6$ & 34{,}783 & 1.413\\
                & 0.15 & 61.50 & $1.712\times10^6$ & 40{,}336 & 1.786\\
                & 0.17 & 64.00 & $1.968\times10^6$ & 43{,}539 & 2.002\\
                & 0.20 & 67.25 & $2.340\times10^6$ & 48{,}020 & 2.275\\
                & 0.25 & 71.50 & $2.920\times10^6$ & 54{,}158 & 2.616\\
\midrule
$\mu_0$ (yr$^{-1}$) & 0.10 & 71.25 & $3.784\times10^6$ & 76{,}884 & 5.471\\
                    & 0.15 & 67.50 & $2.673\times10^6$ & 56{,}838 & 3.181\\
                    & 0.20 & 64.00 & $1.968\times10^6$ & 43{,}539 & 2.002\\
                    & 0.25 & 60.50 & $1.498\times10^6$ & 34{,}216 & 1.334\\
                    & 0.30 & 56.75 & $1.170\times10^6$ & 27{,}419 & 0.929\\
\midrule
$p$ (ind.\,yr$^{-1}$) & $2.5\times10^4$ & 69.50 & $1.157\times10^6$ & 25{,}172 & 2.194\\
                      & $5.0\times10^4$ & 64.00 & $1.968\times10^6$ & 43{,}539 & 2.002\\
                      & $7.5\times10^4$ & 60.00 & $2.596\times10^6$ & 58{,}219 & 1.859\\
                      & $1.0\times10^5$ & 56.75 & $3.109\times10^6$ & 70{,}978 & 1.741\\
\midrule
$u_{\max}$ (yr$^{-1}$) & 0.25 & 55.50 & $1.634\times10^6$ & 51{,}133 & 1.927\\
                       & 0.50 & 64.00 & $1.968\times10^6$ & 43{,}539 & 2.002\\
                       & 0.75 & 67.75 & $2.091\times10^6$ & 41{,}237 & 2.026\\
                       & 1.00 & 70.00 & $2.150\times10^6$ & 40{,}307 & 2.035\\
\bottomrule
\end{tabular}
\end{table}

Finally, faster growth ($K\uparrow$) and a higher harvest cap
($u_{\max}\uparrow$) raise the optimal threshold, since it pays to let fish reach
more valuable classes that are then removed efficiently. Higher mortality
($\mu_0\uparrow$), stronger crowding ($\alpha\uparrow$), heavier discounting
($r\uparrow$) and larger inflow ($p\uparrow$) lower it -- the first three by
raising the opportunity cost of waiting, the last by intensifying density-dependent
suppression. The optimum stays strictly interior; $\mu_0=0.30$ with
$R(E^*)=0.93<1$ illustrates the viability boundary.

\begin{figure}[htbp]
    \centering
    \includegraphics[width=0.95\linewidth]{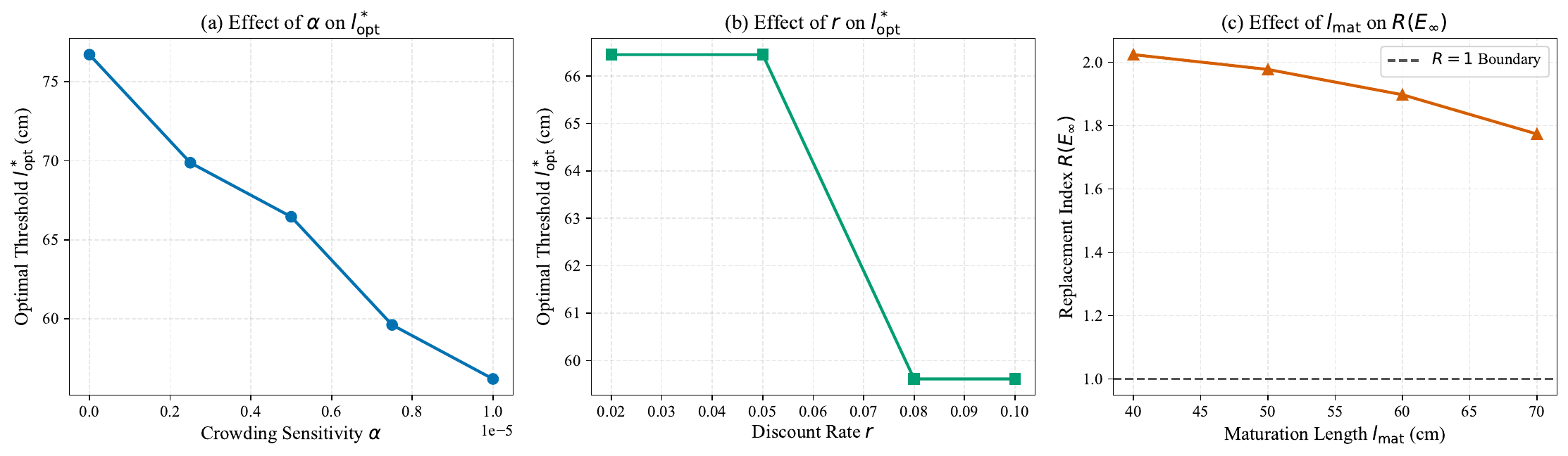}
    \caption{Sensitivity of the revenue-maximizing threshold and biological viability index to varying model parameters. Panel (a) shows the effect of the crowding sensitivity \(\alpha\) on the optimal harvest threshold \(E\). Panel (b) illustrates the impact of the economic discount rate \(r\) on \(E\). Panel (c) displays the response of the intrinsic replacement index \(R(E)\) to changes in the biological maturation length \(l_{\mathrm{mat}}\), with the dashed horizontal line indicating the biological viability boundary \(R=1\). In each run, the specific parameter is varied while keeping all other parameters fixed at their baseline values from Table~\ref{tab:params}.}
    \label{fig:sensitivity_summary}
\end{figure}

\subsection{Summary of numerical findings}

The results show that the proposed size-structured framework captures three
features that are central to managed gadoid fisheries: nonlocal crowding,
size-selective harvesting, and externally maintained recruitment. The numerical
baseline is biologically favorable in the sense that the no-harvest steady
state satisfies
\[
R(E^*)=1.474679>1,
\]
while the threshold policy that maximizes truncated discounted revenue also
remains in the viable region, with
\[
E^*=66.45\ \text{cm},\qquad
R(E^*)=1.977621.
\]
Thus, for the present calibration, the economically preferred harvesting rule
is compatible with the imposed biological viability target.

The sensitivity analysis in Section~\ref{subsec:sensitivity} shows that the
existence of an interior optimal threshold is robust across moderate variations
in system parameters. However, the exact threshold and the replacement index are parameter-sensitive. This confirms that the numerical
case study should be interpreted as a calibrated management illustration rather
than a stock-assessment prediction for a specific cod stock.

\section{Conclusion}\label{sec:conclusion}

This paper developed a nonlinear size-structured fishery model with nonlocal
crowding, exogenous lower-boundary recruitment, and size-selective harvesting.
For the no-harvest case, we derived an explicit stationary profile and a scalar
closure equation for the crowding level, and established an
existence--uniqueness result for the resulting steady state. We then introduced
an intrinsic replacement index appropriate to the externally forced setting,
clarifying that it should be interpreted as a biological viability measure
rather than as a stand-alone persistence criterion.

On the control side, we formulated an infinite-horizon harvesting problem and
derived formal first-order necessary conditions through a Pontryagin-type
variational argument. Under a reduced switching-function analysis and a
single-crossing assumption, this yielded a threshold-type bang--bang harvest
policy. In the numerical case study, the no-harvest baseline was viable, the
replacement index decreased with crowding, and the truncated discounted revenue
was maximized at the interior threshold
\[
E^*=66.45\ \text{cm}.
\]
The corresponding optimal steady state also satisfied the viability target,
showing that, for the present calibration, economic performance and biological
protection are compatible.

From a management perspective, the threshold structure is especially
interpretable.
In Atlantic cod management, size- and catch-based rules are standard. NOAA notes that cod regulations include minimum fish sizes, possession limits, and closed seasons, while Northeast multispecies management materials document stock-level catch-limit frameworks \cite{tallack_regional_2009,noaa_fisheries_atlantic_2026,yu2026age}. 
The present model provides a mathematical rationale for such policies by
showing how a size threshold can emerge from the trade-off between immediate
harvest value and the continuation value of leaving smaller individuals in the
stock.

Three principal limitations, each comes with an explicit remedy.
(i) Recruitment is exogenous, so $\R(E)$ is a viability diagnostic; Section~4.3 and
Proposition~\ref{prop:endo} show how an endogenous Beverton--Holt/Ricker closure
restores the sharp threshold. (ii) The optimality system was formal in the main text; Appendix~\ref{appsec:functional} details well-posedness, existence and a rigorous Pontryagin characterization, and the weak-coupling reduction is controlled by
Proposition~\ref{prop:wcbound}. (iii) The objective is finite-horizon;
Proposition~\ref{prop:trunc} bounds the error and $\widehat J_T$ reduces it below
$0.1\%$.

Future extensions of the current study include:
(1) Close the model with an endogenous stock--recruitment law and
carry out the persistence/bifurcation analysis of Section~\ref{subsec:endo}, including Ricker
overcompensation. (2) Develop a rigorous singular-arc analysis where $\{S=0\}$ has
positive measure. (3) Treat the infinite-dimensional problem via HJB/viscosity
solutions and establish turnpike properties of the discounted structured system.
(4) Replace the deterministic inflow by a stochastically forced boundary
($p$ a random process) and study robust thresholds. (5) Extend to spatial and
multi-species/size-spectrum settings. (6) Add size-dependent harvest costs and
market price feedback to the objective.

Overall, the results support the use of size-structured threshold harvesting as
a transparent management principle in fisheries where recruitment is partly
externally maintained and regulations act directly on fish size. 

\appendix
\section*{Appendix}
\section{Functional setting and rigorous optimality}
\label{appsec:functional}
In this part, we detail the rigorous mathematical functional analytic framework for the well-posedness of the state equation \eqref{eq:state_pmp}, the regularity of the control-to-state map $u\mapsto x^u$, the existence of the optimal control, and the rigorous Pontryagin maximum principle.

\subsection{Function spaces, admissible controls, regularity}
Fix $T\in(0,\infty)$, $\Qt=(0,T)\times(l_0,l_m)$. The state space is
$\X=C([0,T];L^1(l_0,l_m))$ and the admissible set
\[
\Uad=\{u\in L^\infty(\Qt):0\le u\le u_{\max}\ \text{a.e.}\},
\]
convex, bounded, weak-$\ast$ closed, hence weak-$\ast$ sequentially compact in
$L^\infty(\Qt)$. We add the following assumption:
\begin{assumption}[Lipschitz regularity]\label{ass:lip}
On every $[0,M_E]\times[l_0,l_m]$ the maps $g,\mu,\partial_E g,\partial_E\mu$ are
Lipschitz in $(E,l)$, and $\chi\in L^\infty$. The calibrated forms satisfy this on
the a priori compact range $E\in[0,M_E]$ of Theorem~\ref{thm:wp}.
\end{assumption}

\subsection{Well-posedness and a priori bounds}
\begin{theorem}[State well-posedness]\label{thm:wp}
Under Assumptions~\ref{ass:standing} and~\ref{ass:lip}, with $p\in L^\infty_+$
and $\phi\in L^1_+$, every $u\in\Uad$ yields a unique nonnegative mild solution
$x^u\in\X$, with data-dependent bounds (uniform in $u$)
$\|x^u(t)\|_{L^1}\le M_1$, $0\le E^u(t)\le M_E$, and Lipschitz control-to-state
map $\|x^{u_1}-x^{u_2}\|_\X\le L_{cs}\|u_1-u_2\|_{L^1(\Qt)}$.
\end{theorem}
\begin{proof}[Proof sketch]
For frozen $E(\cdot)\in C([0,T];[0,M_E])$ the problem is linear transport; since
$g\ge g_{\min}>0$ the characteristics $\dot L=g(E(t),L)$ are well defined and reach
$l_m$ in finite size-transit time $\tau\le\int_{l_0}^{l_m}d\xi/g_{\min}$. The
Duhamel/characteristic formula gives a unique nonnegative mild solution with
$\|x(t)\|_{L^1}\le\|\phi\|_{L^1}+\int_0^t p$, refined to a uniform bound by exit at
$l_m$ and $\mu+u\ge0$. The nonlocal map $\Gamma:E(\cdot)\mapsto\int\chi x[E]$ is a
contraction on a short interval (Assumption~\ref{ass:lip} + Gr\"onwall); the uniform
bounds let one iterate to cover $[0,T]$. Lipschitz dependence follows from a
Gr\"onwall estimate on the difference, in which $u$ enters only through the bounded
zeroth-order term. Cf.\ \cite{webb_theory_1985,anita_analysis_2000,liu2026ftu}.
\end{proof}

\subsection{Differentiability of the control-to-state map}
\begin{proposition}[Linearized sensitivity]\label{prop:diff}
$u\mapsto x^u$ is G\^ateaux differentiable from $(\Uad,L^\infty(\Qt))$ to $\X$;
for a feasible direction $v$ the sensitivity $z=Dx^u\!\cdot v$ uniquely solves
\[
\partial_t z+\partial_l(g z+x\,\partial_E g\,\zeta)=-(\mu+u)z-x\,\partial_E\mu\,\zeta-v x,
\quad \zeta(t)=\int_{l_0}^{l_m}\chi z\,dl,
\]
with $z(0,\cdot)=0$ and linearized inflow
$g(E,l_0)z(t,l_0)+x(t,l_0)\partial_E g(E,l_0)\zeta(t)=0$.
\end{proposition}

\subsection{Existence of an optimal control}
\begin{theorem}[Existence]\label{thm:exist}
$\sup_{u\in\Uad}J[u]$ is attained by some $\hat u\in\Uad$.
\end{theorem}
\begin{proof}
$J$ is bounded above (Theorem~\ref{thm:wp}, $c\le c_0 l_m^3$). For a maximizing
sequence, $u_n\rightharpoonup^* \hat u\in\Uad$ (weak-$\ast$ compactness). The states are
uniformly bounded in $L^\infty(0,T;L^1)$ and, by the transport structure and
Assumption~\ref{ass:lip}, equi-integrable and time-equicontinuous; an
Aubin--Lions / structured-population compactness argument \cite{anita_analysis_2000}
gives $x^{u_n}\to x^{\hat u}$ strongly in $L^1(\Qt)$. The only $u$-dependence of the
state is the survival factor $\exp(-\int_\gamma u)$ along characteristics, whose
dependence on $u$ is weak-$\ast$-to-strong continuous along the precompact states.
Hence
$J[u_n]=\int e^{-rt}\!\int c\,u_n x^{u_n}\to\int e^{-rt}\!\int c\,\hat u x^{\hat u}=J[\hat u]$,
since $c x^{u_n}\to c x^{\hat u}$ strongly and $u_n\rightharpoonup^* \hat u$.
\end{proof}

\subsection{Rigorous Pontryagin characterization}
\begin{theorem}[Necessary conditions]\label{thm:pmp}
For optimal $(\hat u,\hat x,\hat E)$ the backward adjoint \eqref{eq:adjoint} with
$\lambda(T,\cdot)=0$, $\lambda(\cdot,l_m)=0$ has a unique solution
$\lambda\in L^\infty(\Qt)\cap C([0,T];L^1)$, and
\[
J'[\hat u]v=\int_0^T e^{-rt}\!\int_{l_0}^{l_m}\Phi\,v\,dl\,dt,\qquad \Phi=\hat x(c-\lambda).
\]
Thus $\int e^{-rt}\!\int\Phi(v-\hat u)\le0$ for all $v\in\Uad$, forcing
$\hat u=u_{\max}$ where $\Phi>0$ and $\hat u=0$ where $\Phi<0$, a.e., with arbitrary
values on $\{\Phi=0\}$.
\end{theorem}
\begin{proof}[Proof sketch]
Adjoint well-posedness is Theorem~\ref{thm:wp} run backward. Using
Proposition~\ref{prop:diff} and integration by parts (boundary terms at $l_0,l_m$
vanish by the inflow constraint and $\lambda(\cdot,l_m)=0$, as in
Proposition~\ref{prop:necessary}), the $\zeta$-terms cancel against the adjoint
source, leaving $J'[\hat u]v=\iint e^{-rt}\Phi v$. The variational inequality and
pointwise maximization over $[0,u_{\max}]$ give the rule.
\end{proof}

\bibliography{reference}

\end{document}